\documentclass[sort&compress,3p,12pt]{elsarticle}
\journal{Journal of Statistical Planning and Inference}

\usepackage[T1]{fontenc}
\usepackage[utf8]{inputenc}

\usepackage[UKenglish]{babel}
\usepackage[UKenglish]{isodate}

\usepackage[labelfont=bf]{caption}

\usepackage{amsmath,amsfonts,amssymb,amsthm,booktabs,color,epsfig,graphicx,hyperref,url}

\usepackage{lscape}
\usepackage{mathtools}
\usepackage{commath}
\usepackage{amsmath}
\usepackage{amsthm}
\usepackage{amssymb}
\usepackage{dsfont}
\usepackage{graphicx}
\usepackage{caption}
\usepackage{subcaption}
\usepackage{booktabs} 
\usepackage{adjustbox}
\usepackage{multirow}
\usepackage{pgfplots}
\usepackage{listings}
\usepackage{url}
\usepackage{natbib}
\usepackage{fouriernc}
\usepackage{xcolor}
\usepackage{enumitem}

\usepackage{sectsty}
\usepackage[singlespacing]{setspace}

\newcommand{\ind}{\mathds{1}}

\newcommand{\ew}{\mathrm{E}}
\newcommand{\pr}{\mathrm{Pr}}

\newcommand{\Rbb}{\mathbb{R}}

\newcommand{\Nbb}{\mathbb{N}}

\newcommand{\var}{\mathrm{var}}
\newcommand{\cov}{\mathrm{cov}}
\newcommand{\cf}{\mathcal{F}}

\newcommand{\cp}{\mathcal{P}}

\newcommand{\cb}{\mathcal{B}}

\def\1{{\ind}}

\theoremstyle{plain}
\newtheorem{theorem}{Theorem}

\newtheorem{lemma}{Lemma}
\newtheorem{corollary}{Corollary}

\theoremstyle{definition}
\newtheorem{definition}{Definition}

\begin{document}

\allowdisplaybreaks

\begin{frontmatter}

\title{Simultaneous inference for partial areas under receiver operating curves---with a view towards efficiency}

\author[1]{Maximilian Wechsung \corref{mycorrespondingauthor}}
\author[1]{Frank Konietschke}

\address[1]{Institut für Biometrie und Klinische Epidemiologie, Charité--Universitätsmedizin Berlin,\\ Charitéplatz 1, 10117 Berlin, Germany}

\cortext[mycorrespondingauthor]{Corresponding author. Email address: \url{maximilian.wechsung@posteo.de}}

\begin{abstract}
We propose new simultaneous inference methods for diagnostic trials with elaborate factorial designs. Instead of the commonly used total area under the receiver operating characteristic (ROC) curve, our parameters of interest are partial areas under ROC curve segments that represent clinically relevant biomarker cut-off values. We construct a nonparametric multiple contrast test for these parameters and show that it asymptotically controls the family-wise type one error rate. Finite sample properties of this test are investigated in a series of computer experiments. We provide empirical and theoretical evidence supporting the conjecture that statistical inference about partial areas under ROC curves is more efficient than inference about the total areas. 
\end{abstract}


\end{frontmatter}

\section{Introduction}

Modern medicine can rely on a variety of diagnostic biomarkers ranging from imaging techniques over gene expression measurements to classical laboratory parameters like antibody concentration in a blood sample. Given their ability to detect diseases or disease-related complications at ever earlier stages, they are indispensable in today's clinical practice. Yet, before a biomarker can be admitted for use, its reliability must be ascertained in a series of diagnostic studies so as to satisfy the high standards of contemporary evidence-based medicine.    

Many biomarker measurements produce values on continuous or discrete scales. In order to yield a binary diagnostic decision, the clinical practitioner usually compares the observed marker value to a disease-defining threshold. Subjects with marker values exceeding this threshold will then receive a positive diagnosis. Of course, when we evaluate the marker's viability in a diagnostic study, we do not yet know the optimal choice for the disease-defining threshold. To characterize the marker's performance independently of the specific threshold, we consider the receiver operating characteristic (ROC) curve. For each possible threshold value $t$, the marker has a probability of assigning a positive diagnosis correctly {(true positive rate, TPR$(t)$)} and a probability of doing so incorrectly {(false positive rate, FPR$(t)$)}. With $t$ ranging over $\Rbb$, the ROC curve aggregates all the points $(\mathrm{FPR}(t), \mathrm{TPR}(t))$. The degree to which the resulting curve is attracted by the point $(0,1)$ characterizes the marker's diagnostic reliability---the area under the curve (AUC) is therefore a popular parameter of interest in diagnostic trials.

A typical diagnostic trial examines the reliability of biomarkers at different levels of several experimental factors. For instance, we might want to compare two MRI imaging biomarkers, each of which is interpreted by several different readers. A primary objective is then to uncover the existing main and interaction effects of these factors on the parameter of interest. Statistically, this leads to the problem of testing a family of linear hypotheses \citep{Konietschke2018}. Multiple contrast tests are a tailor-made technique for this type of problem. They are designed to simultaneously test several linear hypotheses while keeping the family-wise type I error rate at the desired level. Being considerably more flexible and less conservative than traditional approaches based on p-value corrections \citep{konietschke2012, hothorn2008simultaneous, bretz2001numerical}, they seem like an ideal method to consider for the statistical analysis of factorial diagnostic studies. It was only recently, though, when Konietschke et al. \citep{Konietschke2018} implemented a multiple contrast test in this context. Based on nonparametric estimators, they constructed multiple contrast tests for AUCs of diagnostic markers in general factorial diagnostic studies. Further developments were due to Blanche et al. \citep{blanche2020closed}. Thus, the possibilities for nonparametric statistical inference about biomarker AUCs in diagnostic trials have greatly improved over the last years.   

Popular though it may be, using the AUC as a performance measure of diagnostic markers has a serious drawback. Even in the simple case of comparing just two markers in terms of their AUCs, the picture can get quite ambiguous if the corresponding ROC curves cross each other \citep{Qin2011}. In this case, the advantage that one marker may enjoy over the other at moderate threshold values $t$ might be offset by the reverse situation at more extreme thresholds, such that both AUCs can be of comparable magnitude even if the ROC curves are rather distinctly shaped. In clinical applications extreme threshold values would most certainly be avoided because they risk an unbearable reduction of the true positive or increase of the false positive probability, and so the marker that is superior for moderate $t$ would be preferable from this perspective. However, the two AUCs being statistically unintelligible, the superior properties of that marker would not be recognized. 

{
The following data example illustrates the just described phenomenon. The example is based on a set of observational data collected at the Department of Radiology at Charité--Universitätsmedizin Berlin. The data set contains the values of three different biomarkers as well as the health status of a total number of 146 patients that were screened for prostate cancer. Since this project is work in progress with a medical publication pending, we cannot reveal the biomarker names nor the raw data. The research question is whether there are any significant differences in the biomarkers' ability to detect prostate cancer. Figure \ref{Fig:example} shows the corresponding ROC curves. The diagram can be divided into two parts: the more central part where a clear order between the ROC curves (green-black-red) is visually apparent, and the marginal part where the three ROC curves are closer together and exhibit a reversed order compared to the central part. In terms of the total AUC, the visible difference in the central part might be countered by the marginal part to the extend that no statistically significant difference in the markers' total AUCs is detectable. Indeed, as we shall see in Section \ref{sec:numerical}, the total areas under the curves are not significantly different. This motivates two intuitive questions. First, are all ROC curve segments equally important for a clinically relevant assessment of the biomarker performance, or is the central part maybe more relevant than the periphery of the diagram? Second, if it turns out that the central part of the ROC curves should indeed be more relevant than the marginal parts, how do we find out whether the visible differences in this part are statistically significant? We can answer both questions only if a useful definition of the term ``clinically relevant ROC curve segment'' and a corresponding statistical parameter are available.}

\begin{figure}[h]
\includegraphics[width=\textwidth]{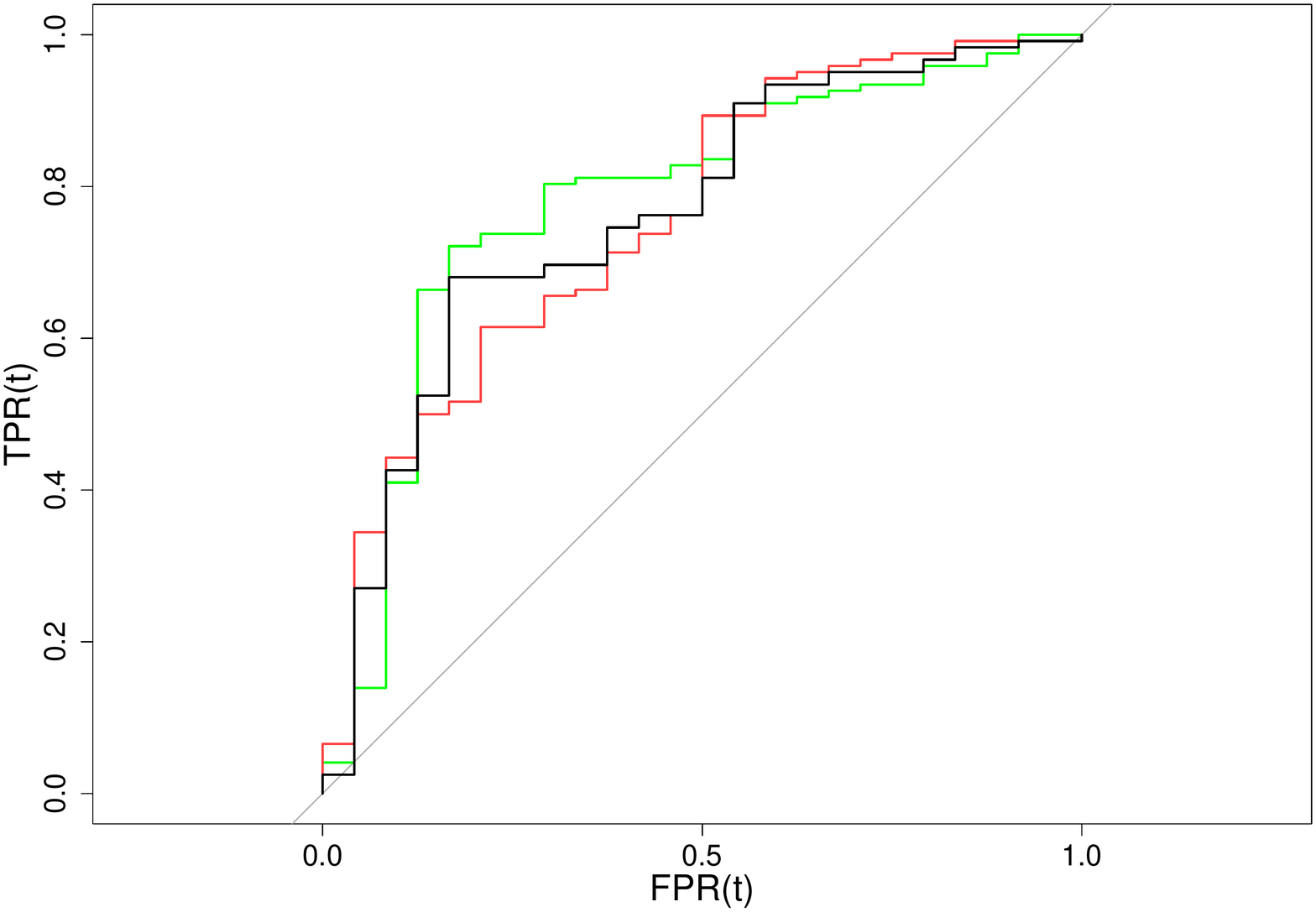}
\caption{ROC curves of three competing diagnostic imaging biomarkers for prostate cancer, based on a retrospective data set obtained at the Department of Radiology, Charité--Universitätsmedizin Berlin.}
\label{Fig:example}
\end{figure}   

Many authors have proposed to define a statistically relevant curve segment via an upper and a lower boundary for the false positive probability \citep{walter2005, mcclish1989, zhang2002, dodd2003, wang2012, ma2013, bandos2017, yang2017}. {The area under this curve segment is then supposed to be an alternative to the AUC as a statistical parameter of interest. We call this parameter \textit{FPR-based partial AUC}. It has several conceptual weaknesses.} First of all, it seems generally rather unintuitive to claim that a threshold value for a biomarker must produce a certain minimum amount of false positive results in order to qualify as clinically relevant. Rather, the FPR lower bound is merely technically motivated, its sole purpose being to implicitly define a TPR lower bound: Since the ROC curve is monotonically increasing, an FPR lower bound automatically defines a TPR lower bound. However, the relationship between a given FPR lower bound and the induced TPR lower bound depends on the concrete shape of the curve, implying two additional reasons why this approach is questionable. First, the ROC curves are unknown. This makes it hard to justify a chosen FPR lower bound because there is generally no way of knowing a priori what the given choice means in terms of the TPR. Second, for a given FPR bound, different ROC curves imply a different TPR bound. For a comparison of several markers, this means that a common FPR lower bound assigned to all involved markers implies different TPR lower bounds for each marker. Why would different markers that are used for the same clinical application have different minimal acceptable true positive rates? 

{These conceptual shortcomings motivate the search for alternative strategies incorporating bounds on both TPR and FPR. An intuitive approach to account for TPR bounds would be to define a TPR-based partial AUC constructed in a completely analogous fashion as the FPR-based partial AUC. While the latter is an area with left and right boundaries in parallel to the ordinate of the ROC diagram, the TPR-based partial AUC's uper and lower boundaries are in parallel to the abcissa \cite[Fig. 1]{carrington2020}. A natural summary measure reconciling a given set of TPR and FPR constraints would then be the arithmetic mean of the corresponding TPR-based and FPR-based partial AUCs. This idea was advocated by Carrington et al. \cite{carrington2020}, who called the so defined parameter \textit{concordant partial AUC}. Conceptually superior to the mere FPR-based partial AUC by virtue of its ability to incorporate TPR constraints as well, the concordant partial AUC still formally requires to specify an FPR lower bound and a TPR upper bound.}

{The approach proposed by Yang et al. \cite{yang2019} circumvents this requirement. 
They solely specify a lower bound for the TPR and an upper bound for the FPR to define a clinically relevant ROC curve segment---the TPR lower bound determines the segment's lower-left end, the FPR upper bound defines the upper-right end. A suitably defined area under the resulting ROC curve segment could then constitute a potential parameter of interest (see Fig. \ref{fig:pAUC}). We refer to this parameter simply as \textit{partial AUC}. Our general objective here shall be to analyze the  partial AUC's viability to serve as a parameter of interest in diagnostic studies and to contrast it with the currently used standard parameter, the total AUC. The concordant partial AUC, a natural alternative to the total AUC in its own right, will not be further discussed in this paper. We defer its treatment to future publications, as an in-depth analysis of its statistical properties is certainly informative and complex enough to fill a paper of its own.}  

The properties of the partial AUC as a statistical parameter, i.e. how it can be inferred upon, are still largely unknown. The only methodological study about this issue is the initial paper by Yang \cite{yang2019} in which the partial AUC was introduced. They proved a large sample result for estimating the difference of two partial AUCs and thereupon constructed a statistical test for comparing two biomarkers. However, statistical tests that are fitting for more complex research questions, e.g. arising in the context of a factorial study, are not available yet. Another point that was brought up in the initial paper was a heuristic argument that the partial AUC should be a more efficient biomarker performance measure than the total AUC in the sense that the former helps to uncover biomarker performance differences with a higher statistical power than the latter. Their computer simulations seem to support this claim, but a theoretical discussion of the phenomenon is still to be given. 

We thus identify two steps that should be undertaken order to qualify the partial AUC as a serious alternative to the total AUC in diagnostic studies: (1) to provide a technique for multiple comparisons of partial AUCs in diagnostic studies with elaborate factorial designs, and (2) a theoretical discussion if and when the partial AUC enjoys better statistical properties than the total AUC. With this article we address both problems. First, we show that the multiple contrast test principle can be successfully applied in connection with a nonparametric estimator of the partial AUC. To this end, we derive the asymptotic distribution of this estimator in a general multivariate setting and show that we can consistently approximate its distribution using Efron's bootstrap. We do not make any assumption on the underlying distribution of the observed marker values or their covariance structure. As a result, we can present an asymptotically correct nonparametric multiple contrast test procedure for simultaneously testing families of arbitrary linear hypotheses about partial AUCs. The total AUC is only a special case of the partial AUC, and our multiple contrast test can therefore as well be applied to the total AUC. This enables us to take on our second goal: a comparison of the total and partial AUC multiple contrast test in terms of their power. We will see that the partial AUC test is indeed often more efficient than the total AUC test in the sense that the former achieves a higher power for a given sample size and effect.

In Section \ref{sec:param_hyp}, we will formally introduce the partial AUC as a statistical parameter and define a multiple testing problem about it. In Section \ref{sec:estimator_test}, we introduce a nonparametric estimator of the partial AUC and construct a multiple contrast test procedure for families of linear hypotheses about the partial AUCs. Section \ref{sec:theoy} states important theoretical results regarding the asymptotic properties of the proposed estimator and test {procedure}, including the estimator's asymptotic distribution and the test {procedure's} ability to asymptotically control the family-wise type I error rate. In Section \ref{sec:numerical}, we present the results of a series of computer experiments and some results on the above mentioned data example to illustrate the properties of the partial AUC and our test procedure. We will also produce empirical evidence in support of the conjecture that inference about the partial AUC is more efficient than inference about the total AUC. In Section 6, we discuss how the results of the computer experiments can be interpreted and explained from a theoretical perspective. We conclude the paper with some remarks about the practical implications of our findings, especially concerning the question of how to choose the clinically relevant ROC curve segment that defines the partial AUC.

\section{The statistical model---parameters and hypotheses}\label{sec:param_hyp}

In this paper we focus on diagnostic trials in which the same population of subjects supplies the marker observations at each combination of factor levels. The typical example is a multi-reader trial comparing several biomarkers, where each subject's marker measurements are interpreted by a fixed number of different readers. A standard assumption in the modelling of diagnostic trials is furthermore the availability of a diagnostic gold standard allowing us to divide the population of participating subjects into two subgroups, one of $\alpha$ non-diseased subjects and one of $\beta$ diseased subjects. If we have a total of $\kappa$ factor level combinations, we end up with the following statistical model. Corresponding to each factor level combination $i \in \{1,\ldots,\kappa\}$, we have a distribution of marker values in the population of diseased subjects, $P^{(1)}_{i}$, and correspondingly a distribution of marker values in the population of non-diseased subjects, $P^{(0)}_{i}$. For each non-diseased subject $r \in \{1,\ldots,\alpha\}$, we observe $\kappa$ different marker values which we aggregate in a $\kappa$-dimensional real vector. This vector is interpreted as the realization of a random vector $\xi_r = \big(\xi_r(1), \ldots, \xi_r(\kappa)\big)^\top \sim P^{(0)}$ whose marginal distributions are $P^{(0)}_1, \ldots, P^{(0)}_\kappa$. Analogously, the observed marker values in a diseased subject $s \in \{1,\ldots,\beta\}$ are realizations of a random vector $\eta_s = \big(\eta_s(1),\ldots,\eta_s(\kappa)\big)^\top \sim P^{(1)}$ with marginals $P^{(1)}_1, \ldots, P^{(1)}_\kappa$. While it is reasonable to assume that $\sigma\{\xi_1,\ldots,\xi_\alpha\}$ and $\sigma\{\eta_1,\ldots,\eta_{_\beta}\}$ are independent \citep{Konietschke2018, Lange2012}, we do not impose any assumption regarding the structure of the distributions $P^{(1)}$ and $P^{(0)}$, making the model entirely nonparametric. We summarize our model assumptions in Definition \ref{def:model}. 

\begin{definition}\label{def:model}
We consider two sequences of random vectors, $\{\xi_r\}_{r \in \Nbb_+}$ and $\{\eta_s\}_{s \in \Nbb_+}$, with $(\xi_r(1), \ldots, \xi_r(\kappa))^\top \sim P^{(0)}$ and $(\eta_s(1), \ldots, \eta_s(\kappa))^\top \sim P^{(1)}$. The cumulative distribution functions of $P^{(0)}$ and $P^{(1)}$ are denoted by $F$ and $G$, respectively. Let $P^{(0)}_i$ and $P^{(1)}_i$ be the marginal distributions of $\xi_r(i)$ and $\eta_r(i)$, respectively, and let $F_i$ and $G_i$ denote the corresponding cumulative distribution functions. We assume that $\sigma\{\xi_r \colon r \in \Nbb_+\}$ and $\sigma\{\eta_s \colon s \in \Nbb_+\}$ are independent. For $\alpha, \beta \in \Nbb_+$, our observations are realizations of $\xi_1,\ldots,\xi_\alpha$ and $\eta_1,\ldots,\eta_{_\beta}$.  \qed
\end{definition}

For each combination of factor levels, we have a ROC curve $\varrho_i$ characterizing the corresponding marker performance. These curves are formally given by the parametrization $\varrho_i = \big\{ \varrho_i(t) = \big(1-F_i(t)\,,\, 1-G_i(t)\big) \colon t \in \Rbb\big\}$. 
Following Yang et al. \cite{yang2019}, we define the clinically relevant segments of the curves $\varrho_i$ by stipulating a maximal acceptable false positive probability $p$ and a minimal acceptable true positive probability $q$. The relevant segment of $\varrho_i$ is then the collection of all points $\varrho_i(t)$ satisfying $1-F_{i}(t) \leq p$ and $1-G_i(t) \geq q$. A graphical illustration is given in Figure \ref{fig:pAUC}. The partial area under $\varrho_i$ that defines the parameter of our interest is the area surrounded by the curves 
\begin{gather*}
 C_{i,1}   = \big\{\big(1-F_i(t),1-G_i(t)\big) \colon F_i^{-1}(1-p) \leq t \leq G_i^{-1}(1-q) \big\} \,,\\ 
C_{i,2} = \big\{ (t,q) \colon 1-F_i(G_i^{-1}(1-q)) \leq t \leq p \big\}\,, \qquad
C_{i,3} = \big\{ (p,t) \colon q \leq t \leq 1-G_i(F_i^{-1}(1-p)) \big\} \,,
\end{gather*}     
implicitly presupposing that $F_i$ and $G_i$ are continuous and strictly increasing.

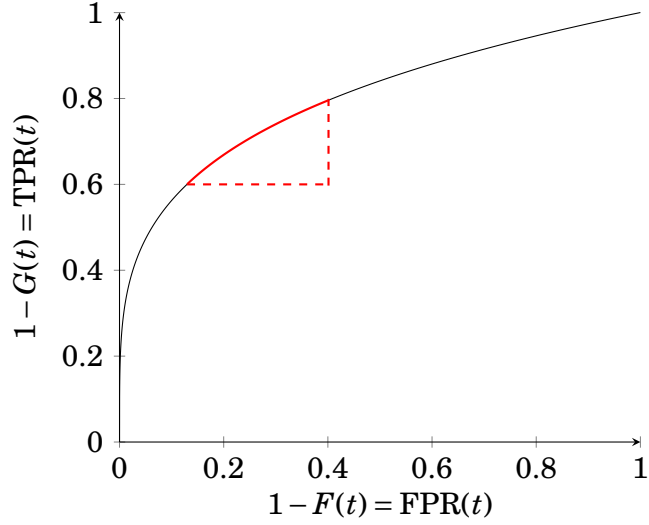
\begin{figure}
 \centering
\begin{tikzpicture}
\begin{axis}[
    axis lines = left,
   xlabel = {$1-F(t) = \mathrm{FPR}(t)$},
    ylabel = {$1-G(t) = \mathrm{TPR}(t)$},
]
\addplot [
    domain=0:(0.6)^4, 
    samples=1000, 
    color=black,
]
{x^(1/4)};
\addplot [
    domain=0.4:1, 
    samples=1000, 
    color=black,
]
{x^(1/4)};
\addplot [
    domain=(0.6)^4:0.4, 
    samples=1000, 
    color=red,
    thick
    ]
    {x^(1/4)};
\addplot [
	dashed, thick,
    domain=(0.6)^4:0.4, 
    samples=1000, 
    color=red,
    ]
    {0.6};
\end{axis}
\draw[red, dashed, thick] (2.75,3.4) -- (2.75,4.55);
\end{tikzpicture}
\caption{The clinically relevant segment (red) of a receiver operating characteristic curve for the FPR upper bound $p=0.4$ and TPR lower bound $q = 0.6$. Together with the curves $C_2$ and $C_3$, depicted as dashed lines, it defines the partial AUC.}
\label{fig:pAUC}
\end{figure}

In some sense the partial area under the curve is a scaled down version of the total area under the curve and therefore its natural extension to boundaries $q > 0$ or $p < 1$. With Green's Theorem we infer that the so defined area is equal to 
\begin{align*}
\int_{(a,b]} \big(G_i(b) - G_i(u)\big) \ \dif F_i(u)\,,
\end{align*} 
with $a = F_i^{-1}(1-p)$ and $b = G_i^{-1}(1-q)$. Using the notion of the generalized quantile function $H^{-1}(p) = \inf\{t \in \Rbb \colon H(t) \geq p\}$ of a distribution function $H$, the expression in the preceding display is also well defined for piece-wise constant functions with jumps. We will therefore use it to formally define multivariate area under the curve parameters. 
In the following, for any $q \in \Nbb_+$, $\{e_1,\ldots,e_q\}$ denotes the canonical basis and $\langle \; \cdot \;,\;\cdot\;\rangle$ the standard inner product in $\Rbb^q$. 

\begin{definition}\label{def:parameter}
{ For $0 < p \leq 1$\,, $0 < q \leq 1$\,,} and $i \in \{1,\ldots,\kappa\}$, we define $a_i = F_{i}^{-1}(1-p)$ and $b_i = G_{i}^{-1}(1-q)$. { If $p,q \in (0,1)$,} we additionally assume that for some $\varepsilon > 0$ each $F_i$ and $G_i$ is continuously differentiable
in $(a_i-\varepsilon,a_i+\varepsilon) \cup (b_i-\varepsilon,b_i+\varepsilon)$. Their derivatives at a point $x$ are denoted $\partial_x F_i(x) =  f_i(x)$ and $\partial_x G_i(x) =  g_i(x)$, respectively. { If $a_i, b_i \in \Rbb$,} we { also} assume that $f_i(a_i) \wedge g_i(a_i) \wedge f_i(b_i) \wedge g_i(b_i) > 0$. The partial area under the curve parameter of the pair of distributions $\big(P^{(0)},P^{(1)}\big)$ is 
\begin{align*}
\theta_{p,q} &= \sum_{i=1}^\kappa e_i \ \int_{(a_i,b_i]} \big(G_{i}(b_i) - G_i(u)\big) \ \dif F_i(u)\,. \tag*{\qed}
\end{align*}
\end{definition}

\noindent
{If we interpret $F_i$ and $G_i$ as functions defined on the extended real line and additionally stipulate that $\inf \emptyset = \infty$, then the boundary cases $q=0$ or $p=0$ are allowed as well. This includes the important special case where $p=1$ and $q=0$, yielding the total area under the curve,}
\begin{align*}
    \theta_{1,0}[P] = \sum_{i=1}^\kappa e_i \ \int_{\overline{\Rbb}} \big(1 - G_{i}(u)\big) \ \dif F_{i}(u)\,.
\end{align*}

Under the assumption that $\sigma\{\eta_1,\ldots,\eta_{_{\beta}}\}$ and $\sigma\{\xi_1,\ldots,\xi_{\alpha}\}$ are independent, a well known representation of the total area under the curve is $\theta_{0,1} = \sum_{i=1}^\kappa e_i \ \pr \big\{ \xi_1(i) < \eta_1(i) \big\}$.
A similar relation holds for the partial area under the curve. Presupposing again independence of the diseased and non-diseased populations, we see that
\begin{align*}
    \theta_{p,q} = \sum_{i=1}^\kappa e_i \ \pr \big\{a_i < \xi_1(i) < \eta_{1}(i) \leq  b_i \big\}\,.
\end{align*}
Furthermore, the partial AUC is invariant under monotone transformations. For a strictly increasing real function $m$, let $F_{m(\xi_1(i))}$ and $G_{m(\eta_1(i))}$ be the distribution function of $m(\xi_1(i))$ and $m(\eta_1(i))$, respectively. Then $F_i(t) = F_{m(\xi_1(i))}(m(t))$ and $G_i(t) = F_{m(\eta_1(i))}(m(t))$, yielding
\begin{align*}
& \pr\big\{ F_{i}^{-1}(1-p) < \xi_{1}(i) < \eta_{1}(i) \leq G_i^{-1}(1-q) \big\} \\
& \hspace{3cm}    = \pr\big\{ F_{m(\xi_{1}(i))}^{-1}(1-p) < m(\xi_{1}(i)) < m(\eta_{1}(i)) \leq G_{m(\eta_{1}(i))}^{-1}(1-q) \big\}\,.
\end{align*}

\section{A nonparametric estimator and multiple contrast test}\label{sec:estimator_test}

What kind of statistical hypotheses about the parameter $\theta_{p,q}$ are we likely to face in a diagnostic study? Perhaps the most frequently posed research questions in the context of factorial experiments concern the presence of main and interaction effects of the involved factors on the parameter of interest. Other examples include Dunnett-type many-to-one comparisons, where we compare each treatment to a fixed reference treatment, or Tukey-type all pairs comparisons, where we conduct a series of pairwise comparisons between treatments.   
All of those examples have in common that they lead to multiple testing problems involving families of linear hypotheses, i.e. they are special cases of the multiple testing problem    
\begin{align}\label{eq:hypotheses}
H_0^i \colon \langle c_i , \theta \rangle = 0  \qquad \text{vs.} \qquad H_1^i \colon \langle c_i , \theta \rangle \neq 0\,, \qquad i\in \{1,\ldots,r\}\,
\end{align}   
where $\theta$ is the vector aggregating the values of the parameter of interest at each factor level and $c_1,\ldots,c_r$ are contrast vectors of according dimension \cite{mukerjee1987,akritas1994,brunner2017}. The matrix $C$ with columns $c_1,\ldots,c_r$ is the contrast matrix characterizing the multiple testing problem. { The form of this testing problem is closely related to Scheffé's method of testing all possible contrasts for sample means ($\theta = \mu$) in an ANOVA setting (see e.g. \cite[p. 81]{dean1999})}. As our focus is on factorial diagnostic trials with the partial AUC as the parameter of interest, we shall investigate the testing problem \eqref{eq:hypotheses} for $\theta = \theta_{p,q}$.

Konietschke et al. \cite{Konietschke2018} discussed this problem for the total AUC (i.e. $\theta = \theta_{1,0}$) and proposed a nonparametric multiple contrast test procedure to deal with it. They proved that the proposed test asymptotically controls the family-wise type I error rate. However, an application to the partial AUC, i.e. the general case $p,q \in [0,1]$, was beyond the scope of their results. { Here} we will close this gap. Based on a nonparametric estimator $\hat{\theta}_{p,q;n}$ of $\theta_{p,q}$, we will construct a family of tests $\{\varphi_{i} \colon i = 1,\ldots,r\}$ designed to simultaneously test any selection of hypotheses $H_0^{i_1},\ldots,H_0^{i_m}$, $1 \leq i_1 < \ldots < i_m \leq r$. Thereafter, we will present a general asymptotic theory of the involved estimator ensuring that our test family asymptotically controls the family-wise type I error rate in the strong sense.    

Each test { of the family $\big\{\varphi_{i} \colon i \in \{1,\ldots,r\}\big\}$} will be based on a nonparametric estimator $\hat{\theta}_{p,q;n}$ of $\theta_{p,q}$, and the family will be constructed { following} the well established multiple contrast test principle \citep{bretz2001numerical, konietschke2012}. The selection of the acceptance regions will be based on quantiles of the joint distribution of the involved test statistics, which we will approximate using the studentized empirical bootstrap. { While critical values based on the normal approximation tend to be rather liberal, the bootstrap} approach often achieves higher order accuracy, { and it appears to be better suited for finite sample sizes} (see e.g. \citep{mammen1992, hall2013}). Working with joint distributions ensures that we do not lose statistical information that may be carried by the structure of the dependencies between the individual test statistics. This promises to make the multiple contrast test approach more powerful than traditional methods based on p-value corrections, which tend to ignore such information \citep{konietschke2012}. Furthermore, it is a suitable basis for the construction of compatible simultaneous confidence intervals.

\begin{definition}\label{def:estimator}
For $\{\alpha_n\}_{n \in \Nbb}, \{\beta_n\}_{n \in \Nbb} \subset \Nbb_+$, $i \in \{1,\ldots,\kappa\}$, $u \in \Rbb$, and $n \in \Nbb$, let $\mathds{F}_{i,n}(u) = \frac{1}{\alpha_n} \sum_{r=1}^{\alpha_n} \ind\{\xi_{r}(i) \leq u\}$ and $\hat{a}_{i,n} = \mathds{F}_{i,n}^{-1}(1-p)$. Analogously we define $\mathds{G}_{i,n}(u)$ and $\hat{b}_{i,n} = \mathds{G}_{i,n}^{-1}(1-q)$. The estimator $\hat{\theta}_{p,q;n}$ of $\theta_{p,q}$ is then
\begin{align*}
\hat{\theta}_{p,q;n}
 &= 
\sum_{i=1}^\kappa e_i \ \int_{(\hat{a}_{i,n},\hat{b}_{i,n}]} \big(\mathds{G}_{i,n}(\hat{b}_{i,n}) - \mathds{G}_{i,n}(u)\big) \ \dif \mathds{F}_{i,n}(u) \,. \tag*{\qed}
\end{align*}  
\end{definition}

Each component $\hat{\theta}_{p,q}(i)$ of $\hat{\theta}_{p,q}$ is a trimmed Mann--Whitney statistic. This can be seen as follows. Let for simplicity $\kappa = 1$, and denote with $\eta_{(r)}$ and $\xi_{(r)}$ the $r$th order statistics of the samples $\{\eta_{1},\ldots,\eta_{_{\beta_n}}\}$ and $\{\xi_{1},\ldots,\xi_{\alpha_n}\}$, respectively. Then
\begin{align}
&\int_{(\hat{a}_{1,n},\hat{b}_{1,n}]} \big(\mathds{G}_{1,n}(\hat{b}_{i,n}) - \mathds{G}_{1,n}(u)\big) \ \dif \mathds{F}_{1,n}(u) \nonumber \\
&= \frac{1}{\alpha_n} \sum_{r=1}^{\alpha_n} \ind\big\{ \xi_{(\lceil n(1-p)\rceil)} < \xi_{r} 
\leq \eta_{(\lceil n(1-q)\rceil)}  \big\} 
\Big[ \frac{1}{\beta_n} \sum_{s=1}^{\beta_n} \ind\big\{ \eta_{s} \leq \eta_{(\lceil n(1-q)\rceil)}\big\} - \ind\big\{ \eta_{s} \leq \xi_{r}\big\} \Big] \nonumber \\
&= \frac{1}{\alpha_n\beta_n} \sum_{r=1}^{\alpha_n} \sum_{s=1}^{\beta_n} \ind\big\{ \xi_{(\lceil n(1-p)\rceil)} < \xi_{r} < \eta_{s} \leq \eta_{(\lceil n(1-q)\rceil)}  \big\} \nonumber \\
&= \frac{1}{\alpha_n\beta_n} \sum_{r=\lceil \alpha_n(1-p) \rceil +1}^{\alpha_n} \sum_{s=1}^{\lceil \beta_n(1-q) \rceil} \ind\big\{ \xi_{(r)} < \eta_{(s)}\big\} \,.
\label{eq:trimmed_MW}
\end{align} 
We see that setting $(p,q)=(1,0)$ results in the usual untrimmed Mann--Whitney estimator of $\theta_{1,0}$,
\begin{gather*}
\hat{\theta}_{1,0;n} = \frac{1}{\alpha_n\beta_n} \sum_{r=1}^{\alpha_n}\sum_{s=1}^{\beta_n} \ind\big\{ \xi_{1,(r)} < \eta_{1,(s)}\big\}\,.
\end{gather*}

\begin{definition}\label{def:cov_est}
For $i,j \in \{1,\ldots,\kappa\}$ and $x,y \in \Rbb$, let 
\begin{align*}
\mathds{F}_n\big(k_{i,j}(x,y)\big) &= \frac{1}{\alpha_n} \sum_{r=1}^{\alpha_n} \ind\big\{ \xi_r({i}) \leq x\,,\,\xi_r({j}) \leq y \big\}\,,  \\
\mathds{G}_n\big(k_{i,j}(x,y)\big) &= \frac{1}{\beta_n}  \sum_{s=1}^{\beta_n}  \ind\big\{ \eta_s({i}) \leq x\,,\,\eta_s({j}) \leq y \big\} \,.
\end{align*}
We define the matrix-valued statistic $\hat{\Sigma}_{p,q;n} = (\hat{\sigma}_{p,q;n}(i,j))_{i,j = 1}^\kappa$ by
\begin{align*}
    \hat{\sigma}_{p,q;n}(i,j) 
    &=
    \frac{\alpha_n+\beta_n}{\beta_n^2\alpha_n}  \sum_{s_1,s_2=1}^{\beta_n} \ind\big\{\hat{a}_{i,n} < \eta_{s_1(i)} \leq \hat{b}_{i,n}\big\} \, \ind\big\{\hat{a}_{j,n} < \eta_{s_2(j)} \leq \hat{b}_{j,n}\big\} \\
    & \hspace{4cm} \big[ \mathds{F}_n\big(k_{i,j}(\eta_{s_1}(i),\eta_{s_2}(j))\big) - \mathds{F}_{i,n}(\eta_{s_1}(i))\mathds{F}_{j,n}(\eta_{s_2}(j)) \big] \\
    &+ \frac{\alpha_n+\beta_n}{\alpha_n^2\beta_n} \sum_{r_1,r_2=1}^{\alpha_n} \ind\big\{\hat{a}_{i,n} < \xi_{r_1(i)} \leq \hat{b}_{i,n}\big\} \, \ind\big\{\hat{a}_{j,n} < \xi_{r_2(j)} \leq \hat{b}_{j,n}\big\} \\
    & \hspace{4cm} \big[ \mathds{G}_n\big(k_{i,j}(\xi_{r_1}(i),\xi_{r_2}(j))\big) - \mathds{G}_{i,n}(\xi_{r_1}(i))\mathds{G}_{j,n}(\xi_{r_2}(j)) \big] \,. \tag*{\qed}
\end{align*}
\end{definition}

It will turn out soon that $|\hat{\theta}_{p,q;n} - \theta_{p,q}| \to 0$ in probability. Hence, it seems sensible to test the hypothesis $H_0^{i} \colon \langle c_i\,,\, \theta_{p,q} \rangle = 0$ by assessing the magnitude of 
\[(\alpha_n + \beta_n)^{1/2} \langle c_i\,,\,\hat{\theta}_{p,q;n}\rangle \big/ \big(\var( \langle c_i\,,\,\hat{\theta}_{p,q;n}\rangle)\big)^{1/2}.\] 
Of course, $\var \big( \langle c_i\,,\,\hat{\theta}_{p,q;n}\rangle \big)$ is unknown, but we will see that $\sum_{r,s=1}^\kappa c_i(r)c_i(s)\hat{\sigma}_{p,q;n}(r,s)$ is a consistent estimator of the asymptotic variance of $(\alpha_n + \beta_n)\, \big(\langle c_i\,,\,\hat{\theta}_{p,q;n}\rangle - \langle c_i\,,\,\theta_{p,q}\rangle\big)$. 

\begin{definition}
The boostrap samples $\{\xi_1^*,\ldots,\xi_{\alpha_n}^*\}$ and $\{\eta_1^*\ldots,\eta_{\beta_n}^*\}$ are drawn with replacement from the original samples $\{\xi_1,\ldots,\xi_{\alpha_n}\}$ and $\{\eta_1,\ldots,\eta_{_{\beta_n}}\}$, respectively, such that bootstrap and original samples are independent. If we apply Definitions \ref{def:estimator} and \ref{def:cov_est} to the bootstrap samples instead of the original samples, we obtain the bootstrap versions of $\hat{\theta}_{p,q;n}$ and $\hat{\Sigma}_{p,q;n}$, to which we refer with $\hat{\theta}^*_{p,q;n}$ and $\hat{\Sigma}^*_{p,q;n}$, respectively.  \qed
\end{definition}
\noindent
At last we define the test family for our multiple testing problem. It is designed to keep the family-wise type I error rate asymptotically below a prescribed level $\delta \in (0,1)$.

\begin{definition}\label{def:test} Let $p,q \in (0,1)$ or $(p,q) = (1,0)$, and let $V_{p,q;n}$ and $V_{p,q;n}^*$ be two diagonal matrices with $\mathrm{diag}(V_{p,q;n}) =  \mathrm{diag}\big(C^\top\hat{\Sigma}_{p,q;n} C\big)$ and $\mathrm{diag}(V_{p,q;n}^*)=\mathrm{diag}(C^\top\hat{\Sigma}_{p,q;n}^* C)$. We define the test statistic $T_n = (\alpha_n + \beta_n)^{1/2} \ V_{p,q;n}^{-1/2}C^\top\hat{\theta}_{p,q;n}$ and the bootstrap statistic $S_n^* = (\alpha_n + \beta_n)^{1/2} \ (V_{p,q;n}^*)^{-1/2} \big(C^\top\hat{\theta}_{p,q;n}^* - C^\top\hat{\theta}_{p,q;n}\big)$. Let
\begin{align*}
q(S_n^*;x)    = \inf \Big\{t \geq 0 \colon \pr \big\{ \| S_n^*\|_\infty \leq t \bigm| \xi_1, \ldots,\xi_{\alpha_n},\eta_1,\ldots,\eta_{_{\beta_n}}  \big\} \geq x \Big\}
\end{align*}
denote the equi-coordinate quantile of $S_n^*$. To test the hypothesis $H_0^i$, we use the decision function
\begin{align*}
\varphi_{i,\delta}(p,q)[\xi_1, \ldots, \xi_{\alpha_n},\eta_1,\ldots,\eta_{_{\beta_n}}]
&= 
\begin{cases}
0 & \text{ if } \qquad \big|\langle T_{n}\,,\, e_i \rangle \big| \leq q(S_n^*;1-\delta) \\
1 & \text{ else.}
\end{cases}
\end{align*}
We further define $\varphi_\delta(p,q) = \max\big\{\varphi_{i,\delta}(p,q) \colon i \in \{1,\ldots,r\}\big\}$ and call it \textit{partial AUC maximum test} if $p \neq  0$ or $q \neq 1$. Otherwise it is called \textit{total AUC maximum test}.  \qed
\end{definition}

\section{Theoretical main results}\label{sec:theoy}

\begin{definition}\label{def:emp_proc}
For $\mathds{T} \subset \overline{\Rbb} \:= \Rbb \cup \{ \pm \infty \}$, the space $(\ell_\infty(\mathds{T}),d_\infty)$ is the set of all bounded functions $g \colon \mathds{T} \to \Rbb$, equipped with the uniform metric, $d_\infty(g,h) := \sup_{t \in \mathds{T}} |g(t) - h(t)|$. Let $(\ell_\infty^{2\kappa}(\mathds{T}),d_\infty^{2\kappa})$ denote the ${2\kappa}$-fold product space $\ell_\infty(\mathds{T}) \times \ldots \times \ell_\infty(\mathds{T})$ equipped with the metric 
\begin{align*}
d_\infty^{2\kappa}\big((x_1,\ldots,x_{2\kappa}),(y_1,\ldots,y_{2\kappa})\big) := \max_{i\in \{1,\ldots,{2\kappa}\}} d_\infty(x_i,y_i)\,.
\end{align*}
Henceforth, $\ell_\infty^{2\kappa}$ shall denote $\ell_\infty^{2\kappa}(\overline{\Rbb})$.
For any $n \in \Nbb_+$, let $X_n$ be an $(\ell_{\infty}^{2\kappa},d^{2\kappa}_\infty)$-valued random element defined by
\begin{align*}
t \mapsto X_n(t) = \sum_{i=1}^{\kappa} \Big[  e_{i} \ \alpha_n^{1/2} \,\big(\mathds{F}_{i,n}(t)- F_i(t) \big)  +  e_{\kappa +i} \ \beta_n^{1/2} \,\big(\mathds{G}_{i,n}(t)- G_i(t) \big) \Big]\,. \tag*{\qed}
\end{align*}
\end{definition}

\begin{theorem}\label{thm:convergence1}
Let $H(x_1,\ldots,x_\kappa,y_1,\ldots,y_\kappa) = F(x_1,\ldots,x_\kappa)\,G(y_1,\ldots,y_\kappa)$ be the cumulative distribution function of $\big(\xi_1(1),\ldots, \xi_1(\kappa),\eta_1(1), \ldots, \eta_1(\kappa)\big)$.
Then, as $n \to \infty$, the sequence $\{X_n\}_{n \in \Nbb}$ converges weakly in $(\ell_\infty^{2\kappa},d^{2\kappa}_\infty)$ to a tight $(\cf-\cb(\ell_\infty^{2\kappa}))$-measurable centered gaussian random element $X_0 = (B_1,\ldots,B_{2\kappa})$ with covariance function
\begin{align*}
\cov\big(B_i(u),B_j(v)\big) =
\begin{cases}
F\big( k_{i,j}(u,v) \big) - F_i(u)F_j(v) \hspace{.75cm} &\text{ if } \quad s,t \in \Rbb; \ i,j \in \{1,\ldots,{\kappa}\}\,,\\
G\big( k_{i,j}(u,v) \big) - G_i(u)G_j(v) \hspace{.75cm} &\text{ if } \quad s,t \in \Rbb; \ i,j \in \{\kappa+1,\ldots,{2\kappa}\},,\\
{ 0} \hspace{.75cm} &\text{ if } \quad s,t \in \Rbb; \ i \in \{1,\ldots,{\kappa}\}\,, j \in \{\kappa+1, \ldots, 2\kappa\}\,,
\end{cases}
\end{align*}
where $k_{i,j}(s,t)$ denotes a vector of suitable dimension satisfying 
\begin{gather*}
    \big\langle k_{i,j}(u,v)\,,\, e_k \big\rangle = 
    \begin{cases}
    u           &\text{ if } \quad k=i \neq j\,, \\ 
    v           &\text{ if } \quad k=j \neq i\,, \\
    u \wedge v  \quad &\text{ if } \quad k = i = j\,, \\
    \infty      & \text{ if } \quad k \notin \{i,j\}\,. 
    \end{cases}
\end{gather*}
\end{theorem}

\begin{definition}
Let $D(\overline{\Rbb})$ be the set of right continuous real functions with left limits, and $\mathds{D}$ the set of pairs $(\vartheta_1,\vartheta_2) \in D(\overline{\Rbb})\times D(\overline{\Rbb})$ such that $\vartheta_1$ has variation bounded by 1. Furthermore, set $\mathds{D}_1 = \mathds{D}^{2\kappa}$, and $\mathds{D}_2 = \mathds{D} \times \Rbb^2$. For any real function $\vartheta$, let $\|\vartheta\|_\infty = \sup_{x \in \mathrm{Dom}(\vartheta)} |\vartheta(x)|$. We consider the norms $\|\cdot\|_{\mathds{D}_1}$ and $\|\cdot\|_{\mathds{D}_2}$ given by 
\begin{align*}
\|(\vartheta_1,\ldots,\vartheta_{2\kappa})\|_{\mathds{D}_1} &= \max_{i \in  \{1, \ldots, 2\kappa\}} \|\vartheta_i\|_\infty\,,\\
\|(\vartheta_1,\vartheta_2,x_1,x_2)\|_{\mathds{D}_2} &= \max \big\{\|\vartheta_1\|_\infty, \, \|\vartheta_2\|_\infty,\, |x_1|,\,|x_2|\big\}\,.
\end{align*}
For $p,q \in (0,1)$ and $i\in \{1,\ldots,\kappa\}$, we define maps $\psi_{p,q}^{(i)} \colon \mathds{D}_1 \to \mathds{D}_2$, $\phi_{p,q} \colon \mathds{D}_2 \to \Rbb$, and $\tau_{p,q} \colon \mathds{D}_1 \to \Rbb$ by
\begin{align*}
\psi_{p,q}^{(i)}(\vartheta_1,\ldots,\vartheta_{2\kappa}) 
&= \big(\vartheta_{i},\vartheta_{\kappa + i}, \vartheta^{-1}_{i}(1-p),\vartheta^{-1}_{\kappa + i}(1-q)\big) \,, \\
\phi_{p,q}(\vartheta,\zeta,x_1,x_2) 
&= \int_{(x_1,x_2]} \big( \zeta(b) - \zeta(u) \big) \ \dif \vartheta(u) \,, \\
\tau_{p,q} &= \big( \phi_{p,q} \circ \psi_{p,q}^{(1)}, \ldots, \phi_{p,q} \circ \psi_{p,q}^{(\kappa)} \big)^\top\,.
\end{align*}
The map $\tau_{1,0} \colon \mathds{D}_1 \to \Rbb$ is defined by
\begin{align*}
\tau_{1,0}(\vartheta_1,\ldots,\vartheta_{2\kappa}) = \sum_{i=1}^\kappa e_i \ \int_{\Rbb} \big( 1 - \vartheta_{\kappa + i}(u) \big) \ \dif \vartheta_{i}(u)\,. \tag*{\qed}
\end{align*}
\end{definition}

\begin{lemma}\label{lemma:hadamard_dif}
For $p,q \in (0,1)$, the function $\tau_{p,q} \colon (\mathds{D}_1,\|\cdot\|_{\mathds{D}_1}) \to (\Rbb^\kappa,\|\cdot\|_\infty)$ is Hadamard differentiable at the point $(F_1,\ldots,F_{\kappa},G_1,\ldots,G_\kappa) \in \mathds{D}_1$, tangentially to the set 
\begin{align*}
\mathds{D}_0 
&:= \big\{(h_1,\ldots,h_{2\kappa}) \in \mathds{D}_{1} \colon \text{for each } i \in \{1,\ldots,\kappa\}, \\ 
&\hspace{4cm} h_{i} \text{ is continuous at } F_{i}^{-1}(1-p), \
              h_{\kappa + i} \text{ is continuous at } G_{i}^{-1}(1-q) \big\}\,.
\end{align*}
Its derivative at $(F_1,\ldots,F_{\kappa},G_1,\ldots,G_\kappa)$ is 
\begin{align*}
&D\tau_{p,q}(F_1,\ldots,F_{\kappa},G_1,\ldots,G_\kappa)[h_1,\ldots,h_{2\kappa}] \\
&\hspace{0.8cm}= \sum_{i=1}^\kappa e_i \Bigg[ 
\int_{\big(F_{i}^{-1}(1-p),G_{i}^{-1}(1-q)\big]} h_{i}(u-) \ \dif G_{i}(u) - \int_{\big(F_{i}^{-1}(1-p),G_{i}^{-1}(1-q)\big]} h_{\kappa + i}(u)  \ \dif F_{i}(u) \Bigg]\,.
\end{align*}
The functional $\tau_{1,0}$ is Hadamard differentiable as well, with derivative
\begin{align*}
&D\tau_{1,0}(F_1,\ldots,F_{\kappa},G_1,\ldots,G_\kappa)[h_1,\ldots,h_{2\kappa}] \\
&\hspace{1cm}    = \sum_{i=1}^\kappa e_i \Bigg[ 
\int_{\Rbb} h_{i}(u-) \ \dif G_{i}(u) - \int_{\Rbb} \big( h_{\kappa + i}(u) - \lim_{x \to -\infty} h_{\kappa + i}(x) \big) \ \dif F_{i}(u) \Bigg]\,.
\end{align*}
\end{lemma}

\begin{theorem}\label{thm:convergence2}
Suppose that $\lim_{n \to \infty} \alpha_n \big/ (\alpha_n + \beta_n) = \lambda \in (0,1)$. Then, as $n \to \infty$, the sequence of random vectors $\big\{ (\alpha_n + \beta_n)^{1/2} \big(\hat{\theta}_{p,q;n}-\theta_{p,q}\big) \big\}_{n \in \Nbb}$ converges weakly to the normally distributed random vector
\begin{align*}
w_{p,q} = D\tau_{p,q}(F_1,\ldots,F_{\kappa},G_1,\ldots,G_\kappa)\Big[\frac{B_1}{\lambda^{1/2}},\ldots,\frac{B_\kappa}{\lambda^{1/2}},\frac{B_{\kappa+1}}{(1-\lambda)^{1/2}},\ldots, \frac{B_{2\kappa}}{(1-\lambda)^{1/2}}\Big]\,.
\end{align*}
The limit $w_{p,q}$ has mean zero and the elements of its covariance matrix $\Sigma_{p,q}$ are
\begin{align*}
\sigma_{p,q}(i,j) &= \lambda^{-1}  \iint_{D}  \cov\big(B_{i}(u),B_{j}(v)\big)  \ P^{(1)}_{i}\otimes P^{(1)}_{j}(\dif u, \dif v) \\
    & \hspace{2.8cm}+ (1-\lambda)^{-1} \iint_{D} \cov \big(B_{\kappa + i}(u),B_{\kappa + j}(v)\big) \ P^{(0)}_{i}\otimes P^{(0)}_{j}(\dif u, \dif v)  \,,
\end{align*}
where $D = [a_i,b_i)\times[a_j,b_j)$, $i,j \in \{1,\ldots,\kappa\}$\,.
\end{theorem}

\begin{corollary}
The statistic $\hat{\sigma}_{p,q;n}(i,j)$ is a consistent estimator of $\sigma_{p,q}(i,j)$. 
\end{corollary}

\noindent
Slightly altered, Theorems \ref{thm:convergence1} and \ref{thm:convergence2} are valid without the assumption that $\sigma\{\xi_r \colon r \in \Nbb_+ \}$ and $\sigma\{\eta_s \colon s \in \Nbb_+ \}$ are independent. { In this case, 
\begin{align*}
\cov\big(B_i(u),B_j(v)\big) = H(k_{i,j}(u,v))-F_i(u)G_j(v) \quad &\text{ if } \quad s,t \in \Rbb; \ i \in \{1,\ldots,{\kappa}\}\,, j \in \{\kappa+1, \ldots, 2\kappa\}\,,
\end{align*}
and} the covariance $\sigma_{p,q}(i,j)$ has the additional addend
\begin{align*}
       - [\lambda (1-\lambda)]^{-1/2} \Big[ \iint_{D} \cov \big(B_{i}(u),B_{\kappa+j}(v) \big) \ P^{(1)}_{i}\otimes P^{(0)}_{j}(\dif u, \dif v) 
     + \iint_{D} \cov \big(B_{\kappa + i}(u),B_{j}(v) \big) \ P^{(0)}_{i}\otimes P^{(1)}_{j}(\dif u, \dif v) \Big] \,.
\end{align*}

It is already well known that the sequence $\big\{ (\alpha_n + \beta_n)^{1/2} \big(\hat{\theta}_{1,0}-\theta_{1,0}\big) \big\}_{n \in \Nbb}$ converges weakly to a normally distributed random vector. As a side result we obtain an explicit form of that limit,
\begin{align*}
 w_{1,0} = D\tau_{1,0}(F_1,\ldots,G_{\kappa})\big[\lambda^{-1/2}B_1,\ldots,\lambda^{-1/2}B_\kappa,(1-\lambda)^{-1/2}B_{\kappa+1},\ldots,(1-\lambda)^{-1/2}B_{2\kappa}\big].    
\end{align*}
The corresponding covariance matrix $\Sigma_{1,0}$ has the same form as $\Sigma_{p,q}$ except that $D = \Rbb^2$.

\begin{lemma}\label{thm:boostrap}
Under the conditions of Theorem \ref{thm:convergence2}, the sequence $\big\{(\alpha_n+\beta_n)^{1/2}\big(\hat{\theta}^*_{p,q;n}- \hat{\theta}_{p,q;n}\big)\big\}_{n \in \Nbb}$ converges conditionally in distribution to $w_{p,q}$, i.e. 
\begin{align*}
\sup_{h \in \mathrm{BL}_1(\Rbb^\kappa)} \Big| \ew \Big[ h \big( (\alpha_n+\beta_n)^{1/2}\big( \hat{\theta}^*_{p,q;n} - \hat{\theta}_{p,q;n}\big) \big) \bigm| \xi_1,\ldots,\xi_{\alpha_n},\eta_1,\ldots,\eta_{\beta_n} \Big] - \ew \, h(w_{p,q})  \Big| \to 0 
\end{align*}
in outer probability as $n \to \infty$, where $\mathrm{BL}_1(\ell_\infty^\kappa(\Rbb))$ denotes the set of bounded Lipschitz functions $\Rbb^\kappa \to [-1,1]$ with Lipschitz constant at most $1$.
\end{lemma} 

\noindent
Based on the Glivenko--Cantelli Theorem for Efron's Boostrap \citep[Theorem 3.2]{wellner2001}, we can additionally infer that the bootstrap version of the covariance estimator is consistent.

\begin{lemma}\label{lemma:boostrap_var}
Let $\hat{\sigma}^*_{p,q;n}(i,j)$ denote the bootstrap version of
$\hat{\sigma}_{p,q;n}(i,j)$. Then, for almost all data generating sequences $\{\xi_{\alpha_n}\}_{n \in \Nbb}$\,, $\{\eta_{\beta_n}\}_{n \in \Nbb}$ and every $\varepsilon > 0$,
\begin{align*}
\pr \Big\{ \big| \hat{\sigma}^*_{p,q;n}(i,j) - \hat{\sigma}_{p,q;n}(i,j) \big| > \varepsilon \bigm| \xi_1,\ldots,\xi_{\alpha_n}, \eta_1,\ldots,\eta_{\beta_n} \Big\} \to 0\,,
\end{align*}  
as $n \to \infty$.
\end{lemma}

\begin{theorem}\label{lemma_test}
Let $p,q \in (0,1)$ or $(p,q) = (1,0)$. Then, under the conditions of Theorem \ref{thm:convergence2}, the family of tests $\big\{\varphi_{i,\delta} \colon i \in \{1,\ldots,r\}\big\}$ for the testing problem \eqref{eq:hypotheses} asymptotically controls the family wise error rate in the strong sense; i.e. for any subset $\{i_1,\ldots,i_m\} \subset \{1,\ldots,r\}$, if $H_0^{i_1} \wedge \ldots \wedge H_0^{i_m}$ is true, then  
\begin{align*}
\lim_{n \to \infty} \pr \Big( \bigcup_{j=1}^m \big\{ \varphi_{i_j,\delta}(\xi_1,\ldots,\xi_{\alpha_n},\eta_1,\ldots,\eta_{_{\beta_n}}) = 1 \big\} \Big)  \leq \delta\,.
\end{align*}
The test $\varphi_\delta$ for the global hypotheses
$H_0 \colon C^\top\theta_{p,q} = 0$ vs. $H_1 \colon C^\top\theta_{p,q} \neq 0$ is asymptotically of size $\delta$.
\end{theorem}

{ }

\section{Computer experiments and data example}\label{sec:numerical}

\subsection{Material and methods}

To illustrate and corroborate our theoretical findings, we conducted a series of computer experiments. { In the first set of simulations, we focused on the asymptotic behaviour of the maximum test $\varphi_\delta$ under the global null hypothesis. 
The remaining two sets of simulations were devoted to comparing the power of the total and partial AUC maximum tests against comparable alternatives.} In all instances we approximated the bootstrap quantiles based on 2000 independent repetitions. After 10,000 simulation runs, we computed the average errors. All simulations were carried out using the \textsf{R} programming language, version 4.1.0 \citep{R}. The code is available from the authors upon request.

\paragraph{Type I error simulations.} 

{ We simulated a balanced diagnostic trial (i.e. $\alpha = \beta$) with one-way layout and three factor levels. Here, we chose a Tuckey-type contrast matrix for testing all possible pairs of factor levels. We were mainly interested in the necessary sample sized that is needed to attain the aspired type I error rate of $\delta = 0.05$.}
We set $\kappa=3$ and defined the joint distribution function of the observations by
$ F(u_1,\ldots,u_6) = \mathcal{C}\big(F_1(u_1), F_2(u_2),F_3(u_3),G_1(u_4),G_2(u_5),G_3(u_6) \big)$,
where $\mathcal{C}$ is a Gaussian copula with a Spearman correlation matrix
\begin{scriptsize}
\begin{align*}
R_s = 
\begin{pmatrix} 
  1.00 & 0.79 & 0.38 & 0.00 & 0.00 & 0.00 \\
  0.79 & 1.00 & 0.79 & 0.00 & 0.00 & 0.00 \\
  0.38 & 0.79 & 1.00 & 0.00 & 0.00 & 0.00 \\
  0.00 & 0.00 & 0.00 & 1.00 & 0.79 & 0.38 \\
  0.00 & 0.00 & 0.00 & 0.79 & 1.00 & 0.79 \\
  0.00 & 0.00 & 0.00 & 0.38 & 0.79 & 1.00
\end{pmatrix}\,.
\end{align*}
\end{scriptsize}

{Note that the two blocks of zeros reflect the assumed independence between diseased and non-diseased population, cf. Definition \ref{def:model}.}
We chose $F_1$, $G_1$, $F_2$, $G_2$, $F_3$, $G_3$ as the cumulative distribution functions of $\mathcal{N}(0,1)$, $\mathcal{N}(0.5,1)$, $\exp\big(\mathcal{N}(0,1)\big)$, $\exp\big( \mathcal{N}(0.5,1)\big)$, $\big[1+\exp\big(-\mathcal{N}(0,1)\big)\big]^{-1}$, $\big[1+\exp\big(-\mathcal{N}(0.5,1)\big)\big]^{-1}$, respectively. We tested the global hypothesis $H_0 \colon C^\top\theta_{p,q} = 0$ for a $3 \times 3$ Tukey-type all pair comparison contrast matrix, i.e. the contrast vectors are $c_1 = (1,-1,0)^\top$, $c_2 = (0,1,-1)^\top$, and $c_3 = (1,0,-1)^\top$. The type I error rates of $\varphi_\delta$ were computed for different pairs $(p,q) \in \{0,0.2,\ldots,1\}\times\{0,0.2,\ldots,1\}$ and are displayed in Table \ref{table:alpha}. 

\paragraph{Power comparison}

{We compared the power of two partial AUC maximum tests, $\varphi_\delta(0.6,0.4)$ and $\varphi_\delta(0.8,0.6)$, and the total AUC maximum test $\varphi_\delta(1,0)$ against the fixed alternatives $H_1 \colon C^\top \theta_{0.6,0.4} = \lambda \, v$, $H_1' \colon C^\top\theta_{0.8,0.6} = \lambda \, v$ and $H_1'' \colon C^\top\theta_{1,0} = \lambda \, v$, respectively, where the scalars $\lambda \in \Rbb$ determine the size of the alternatives and the unit vector $v$ their direction. We considered different settings, varying in $C$, $v$, and $\lambda$. In all scenarios, the level of significance was $\delta = 0.05$.

\begin{enumerate}[label=(\alph*),wide,labelindent=0pt]
\item For the first set of power comparisons, the scenario was almost as in the type I error simulations. The contrast matrix and the marginal distributions $F_1,F_2,F_3,G_1,G_2$ were chosen as above. This choice implies that the alternatives $H_1$, $H_1'$, and $H_1''$ all point into the direction $v = 2^{-1/2}(-1,1,0)^\top$. The parameter $\mu$ of the distribution $G_3 \sim \big[1+\exp\big(-\mathcal{N}(\mu,1)\big)\big]^{-1}$ was then tuned such that $C^\top\theta \approx \lambda \, v$ for the respectively aspired effect size $\lambda$.
The copula $\mathcal{C}$ to define the joint distribution was chosen as in the type I error simulation. The results of this set of power comparisons are displayed in Table \ref{table:beta}

\item Here, we simulated a balanced two-factorial crossed design. The first factor ($A$) has three levels, the second factor ($B$) has two levels, so that $\kappa = 6$. We chose the contrast matrix such that $H_0 \colon C^\top\theta_{p,q} = 0$ was equivalent to testing the hypothesis of no interaction effect $A\times B$, i.e. 
\begin{align*}
C = \begin{pmatrix}
\hphantom{-}\frac{2}{3} & -\frac{1}{3} & -\frac{1}{3} \\
 -\frac{1}{3} & \hphantom{-} \frac{2}{3} & -\frac{1}{3} \\
 -\frac{1}{3} & -\frac{1}{3} &\hphantom{-} \frac{2}{3}
\end{pmatrix}
\otimes 
\begin{pmatrix}
\hphantom{-}\frac{1}{2} & -\frac{1}{2}  \\
 -\frac{1}{2} & \hphantom{-} \frac{1}{2}
\end{pmatrix}\,,
\end{align*}
where $\otimes$ denotes the Kronecker product (cf. Brunner et al. \cite[p. 268]{brunner2017}). The underlying data distribution was again defined via a gaussian copula, $\tilde{\mathcal{C}}(F_1,\ldots,F_6,G_1,\ldots,G_6)$, with spearman correlation matrix
\begin{scriptsize}
\begin{align*}
R_s = 
\begin{pmatrix}
  1.00  & 0.86 & 0.73 & 0.61 & 0.45 & 0.36 & 0     & 0     & 0     & 0     & 0     & 0 \\
  0.86 & 1.00  & 0.86 & 0.73 & 0.61 & 0.45 & 0     & 0     & 0     & 0     & 0     & 0 \\
  0.73 & 0.86 & 1.00  & 0.86 & 0.73 & 0.61 & 0     & 0     & 0     & 0     & 0     & 0 \\
  0.61 & 0.73 & 0.86 & 1.00  & 0.86 & 0.73 & 0     & 0     & 0     & 0     & 0     & 0 \\
  0.45 & 0.61 & 0.73 & 0.86 & 1.00  & 0.86 & 0     & 0     & 0     & 0     & 0     & 0 \\
  0.36 & 0.45 & 0.61 & 0.73 & 0.86 & 1.00  & 0     & 0     & 0     & 0     & 0     & 0 \\
 	0   & 0     & 0     & 0     & 0     & 0     &   1.00  & 0.86 & 0.73 & 0.61 & 0.45 & 0.36  \\
    0   & 0     & 0     & 0     & 0     & 0     &   0.86 & 1.00  & 0.86 & 0.73 & 0.61 & 0.45 \\
    0   & 0     & 0     & 0     & 0     & 0     &   0.73 & 0.86 & 1.00  & 0.86 & 0.73 & 0.61 \\
    0   & 0     & 0     & 0     & 0     & 0     &   0.61 & 0.73 & 0.86 & 1.00  & 0.86 & 0.73 \\
    0   & 0     & 0     & 0     & 0     & 0     &  0.45 & 0.61 & 0.73 & 0.86 & 1.00  & 0.86  \\
    0   & 0     & 0     & 0     & 0     & 0     & 0.36 & 0.45 & 0.61 & 0.73 & 0.86 & 1.00   \\
\end{pmatrix}\,.
\end{align*}
\end{scriptsize}
We defined $F_1,\ldots,F_6,G_1,\ldots,G_6$ as the cumulative distribution functions of $\big[1+\exp\big(-\mathcal{N}(0,1)\big)\big]^{-1}$, $\exp\big(\mathcal{N}(0,1)\big)$, $\mathcal{N}(0,1)$, $\big[1+\exp\big(-\mathcal{N}(0,1)\big)\big]^{-1}$, $\exp\big(\mathcal{N}(0,1)\big)$, $\mathcal{N}(0,1)$,$\big[1+\exp\big(-\mathcal{N}(0,1)\big)\big]^{-1}$,$\exp\big(\mathcal{N}(0,1)\big)$, $\mathcal{N}(0,1)$, $\big[1+\exp\big(-\mathcal{N}(0,1)\big)\big]^{-1}$,$\exp\big(\mathcal{N}(0,1)\big)$, $\mathcal{N}(\mu,1)$, respectively. Due to this choice and the shape of $C$, the alternatives $H$, $H'$, $H''$ all pointed into the direction $v = 12^{-1/2} (1,-1,1,-1,-2,2)^\top$.
\end{enumerate} 
}

{
\paragraph{Data example} 

We tested the null hypothesis
\begin{align*}
H_0 \colon 
\begin{pmatrix}
-1 & \hphantom{-}1 & \hphantom{-}0 \\
-1 & \hphantom{-}0        & \hphantom{-}1 \\
\hphantom{-}0  &     -1 & \hphantom{-}1
\end{pmatrix}
\begin{pmatrix}
\theta_{p,q}(1) \\
\theta_{p,q}(2) \\
\theta_{p,q}(3)
\end{pmatrix}
= 0\,.
\end{align*} 
for different values of $p$ and $q$. The p-values were computed based on a bootstrap statistic. The bootstrap distribution was approximated based on 2000 iterations. The results are presented in Table \ref{table:data}. The variations in $p$ and $q$ also induce a multiple testing problem. This multiplicity is not accounted for by our multiple contrast test procedure. We therefore also present p-values that are adjusted means of the Bonferroni--Holm method.

\subsection{Results}
Our type I error simulations demonstrate that $\varphi_\delta$ rapidly attains the desired size of $0.05$. However, the necessary sample size for this observation varied significantly between the different settings. {In some cases the test is rather accurate for small sample sizes around $n = 30$, whereas for other constellations $n=80$ observations per group are necessary. We believe that two factors are mainly responsible for this variation, which we will lay out in the next section.}

As to the power simulations, the type II error rates were also quite sensitive to the choice of $p$ and $q$. In any case, we observe that the considered partial AUC maximum tests detected the deviances from the global null hypothesis more frequently than the total AUC test, at least for sample sizes sufficiently large for the tests to attain their nominal sizes as indicated by Table \ref{table:alpha}.

In our data example, the statistical evidence for a difference in the biomarkers' performances varied substantially with different choices for $p$ and $q$. Even though no adjusted p-value was below the critical threshold of $0.05$, we would argue that the exploratory character of the data evaluation in combination with our statistical results would justify the planning of a prospective diagnostic study to examine the performance differences of the three biomarkers. For such a study, the clinically relevant curve segment should be defined in advance by the medical investigator.

\begin{table}[h!]
\captionsetup{font=footnotesize}
\begin{minipage}{\textwidth}
\begin{adjustbox}{width=0.8\textwidth,center}
\begin{tabular}{ccccccccccccccccccccccc}\toprule[1.5pt]
         &&\multicolumn{19}{c}{$n$} \\ \cmidrule{3-23}
 $(p\,,\,q)$   &&  30      && 40     &&    50   && 60     && 70     && 80       && 90       && 100    && 110    &&   120 && 130\\ \midrule[1.2pt]
 (0.2\,,\,0.0) &&  0.0088  && 0.0180 && 0.0254  && 0.0307 && 0.0345 && 0.0379 && 0.0388 && 0.0393 && 0.0414 && 0.0461 && 0.0430 \\
 (0.4\,,\,0.0) &&  0.0272  && 0.0378 && 0.0414  && 0.0410 && 0.0404 && 0.0441 && 0.0488 && 0.0443 && 0.0499 && 0.0477 && 0.0469\\
 (0.4\,,\,0.2) &&  0.0119  && 0.0207 && 0.0309  && 0.0376 && 0.0409 && 0.0403 && 0.0448 && 0.0428 && 0.0476 && 0.0466 && 0.0451\\
 (0.6\,,\,0.0) &&  0.0400  && 0.0428 && 0.0435  && 0.0441 && 0.0502 && 0.0478 && 0.0465 && 0.0435 && 0.0454 && 0.0441 && 0.0504\\
 (0.6\,,\,0.2) &&  0.0308  && 0.0215 && 0.0327  && 0.0310 && 0.0376 && 0.0418 && 0.0420 && 0.0480 && 0.0422 && 0.0427 && 0.0486\\
 (0.6\,,\,0.4) &&  0.0140  && 0.0222 && 0.0273  && 0.0354 && 0.0416 && 0.0434 && 0.0440 && 0.0465 && 0.0454 && 0.0458 && 0.0468\\ 
 (0.6\,,\,0.6) &&  0.0168  && 0.0200 && 0.0340  && 0.0334 && 0.0372 && 0.0402 && 0.0416 && 0.0401 && 0.0442 && 0.0463 && 0.0508\\ 
 (0.8\,,\,0.0) &&  0.0392  && 0.0443 && 0.0448  && 0.0467 && 0.0484 && 0.0472 && 0.0482 && 0.0526 && 0.0501 && 0.0459 && 0.0476\\ 
 (0.8\,,\,0.2) &&  0.0397  && 0.0436 && 0.0501  && 0.0482 && 0.0482 && 0.0485 && 0.0478 && 0.0536 && 0.0510 && 0.0511 && 0.0507\\ 
 (0.8\,,\,0.4) &&  0.0406  && 0.0431 && 0.0479  && 0.0508 && 0.0465 && 0.0505 && 0.0475 && 0.0501 && 0.0518 && 0.0480 && 0.0502\\ 
 (0.8\,,\,0.6) &&  0.0442  && 0.0420 && 0.0471  && 0.0496 && 0.0503 && 0.0519 && 0.0460 && 0.0498 && 0.0486 && 0.0487 && 0.0490\\ 
 (0.8\,,\,0.8) &&  0.0435  && 0.0463 && 0.0484  && 0.0480 && 0.0478 && 0.0508 && 0.0477 && 0.0484 && 0.0509 && 0.0507 && 0.0498\\
 (1.0\,,\,0.0) &&  0.0409  && 0.0457 && 0.0447  && 0.0470 && 0.0498 && 0.0553 && 0.0461 && 0.0503 && 0.0514 && 0.0536 && 0.0461\\ 
 (1.0\,,\,0.2) &&  0.0431  && 0.0486 && 0.0487  && 0.0482 && 0.0469 && 0.0505 && 0.0506 && 0.0504 && 0.0479 && 0.0459 && 0.0467\\ 
 (1.0\,,\,0.4) &&  0.0377  && 0.0437 && 0.0458  && 0.0433 && 0.0462 && 0.0459 && 0.0493 && 0.0497 && 0.0493 && 0.0463 && 0.0520\\ 
 (1.0\,,\,0.6) &&  0.0323  && 0.0413 && 0.0417  && 0.0448 && 0.0460 && 0.0458 && 0.0491 && 0.0469 && 0.0479 && 0.0544 && 0.0494\\ 
 (1.0\,,\,0.8) &&  0.0079  && 0.0180 && 0.0219  && 0.0292 && 0.0336 && 0.0388 && 0.0382 && 0.0412 && 0.0466 && 0.0391 && 0.0483\\
\bottomrule[1.5pt]
\end{tabular}
\end{adjustbox}
\caption{Type I error rates for the partial AUC maximum test with Tukey-type contrast matrix.}\label{table:alpha}
\end{minipage}
\end{table}

\begin{table}[h!]
\captionsetup{font=footnotesize}
\begin{minipage}{\textwidth}
\begin{adjustbox}{width= 0.8\textwidth, center}
\begin{tabular}{ccccccccccccccccccc}\toprule[1.5pt]
                       &&               &&                              \multicolumn{15}{c}{$n$}                        \\ \cmidrule{5-19}
  && $\lambda$ &&  30   &&   40     &&    50   && 60     && 70     && 80       && 90   && 100  \\ \midrule[1.2pt]
 $\varphi_\delta(0.8,0.6)$ &&   0.107   && 0.6297  && 0.8504 && 0.9472  && 0.9835 && 0.9953 && 0.9986 && 0.9995 && 0.9999 \\	
 $\varphi_\delta(0.8,0.6)$ &&   0.061   && 0.1858  && 0.3352 && 0.4736  && 0.6038 && 0.6946 && 0.7745 && 0.8287 && 0.8768 \\ 
 $\varphi_\delta(0.8,0.6)$ &&   0.021   && 0.0322  && 0.0493 && 0.0730  && 0.0913 && 0.1132 && 0.1273 && 0.1449 && 0.1639 \\ \midrule
 $\varphi_\delta(0.6,0.4)$ &&   0.107   && 0.4910  && 0.7447 && 0.8713  && 0.9410 && 0.9743 && 0.9887 && 0.9967 && 0.9989 \\ 
 $\varphi_\delta(0.6,0.4)$ &&   0.061   && 0.1416  && 0.2572 && 0.3659  && 0.4669 && 0.5647 && 0.6465 && 0.7040 && 0.7549 \\ 		 
 $\varphi_\delta(0.6,0.4)$ &&   0.021   && 0.0285  && 0.0427 && 0.0607  && 0.0824 && 0.0901 && 0.1035 && 0.1170 && 0.1283 \\ \midrule
 $\varphi_\delta(1,0)$     &&   0.107   && 0.4397  && 0.5985 && 0.7420  && 0.8314 && 0.9111 && 0.9472 && 0.9695 && 0.9843 \\
 $\varphi_\delta(1,0)$     &&   0.061   && 0.1633  && 0.2111 && 0.2714  && 0.3387 && 0.3871 && 0.4454 && 0.4969 && 0.5528 \\ 
 $\varphi_\delta(1,0)$     &&   0.021   && 0.0548  && 0.0651 && 0.0693  && 0.0769 && 0.0811 && 0.0878 && 0.0934 && 0.0976  \\ 
\bottomrule[1.5pt]
\end{tabular}
\end{adjustbox}
\caption{Empirical power comparison of total and partial AUC maximum tests with Tukey-type contrast matrix for different effect sizes.} \label{table:beta}
\end{minipage}
\end{table}  

\begin{table}[h!]
\captionsetup{font=footnotesize}
\begin{minipage}{\textwidth}
\begin{adjustbox}{width= 0.8\textwidth, center}
\begin{tabular}{ccccccccccccccccccc}\toprule[1.5pt]
                       &&               &&                              \multicolumn{15}{c}{$n$}                        \\ \cmidrule{5-19}
  && $\lambda$ &&  30   &&   40     &&    50   && 60     && 70     && 80       && 90   && 100  \\ \midrule[1.2pt]
 $\varphi_\delta(0.8,0.6)$ &&   0.054  && 0.6012  && 0.7970 && 0.9028  && 0.9563 && 0.9816 && 0.9921 && 0.9973 && 0.9989 \\	
 $\varphi_\delta(0.8,0.6)$ &&   0.035  && 0.2877  && 0.4244 && 0.5493  && 0.6556 && 0.7414 && 0.8086 && 0.8616 && 0.8980 \\ 
 $\varphi_\delta(0.8,0.6)$ &&   0.018  && 0.0888  && 0.1298 && 0.1677  && 0.2027 && 0.2399 && 0.2744 && 0.2967 && 0.3354 \\ \midrule
 $\varphi_\delta(0.6,0.4)$ &&   0.052  && 0.4871  && 0.6859 && 0.8090  && 0.8905 && 0.9393 && 0.9707 && 0.9824 && 0.9914 \\ 
 $\varphi_\delta(0.6,0.4)$ &&   0.035  && 0.2302  && 0.3551 && 0.4518  && 0.5551 && 0.6354 && 0.7125 && 0.7584 && 0.8169 \\ 		 
 $\varphi_\delta(0.6,0.4)$ &&   0.018  && 0.0753  && 0.1062 && 0.1336  && 0.1684 && 0.1965 && 0.2228 && 0.2463 && 0.2788 \\ \midrule
 $\varphi_\delta(1,0)$     &&   0.055  && 0.3764  && 0.5604 && 0.6902  && 0.7925 && 0.8751 && 0.9226 && 0.9505 && 0.9693 \\
 $\varphi_\delta(1,0)$     &&   0.035  && 0.1614  && 0.2447 && 0.3224  && 0.3990 && 0.4754 && 0.5355 && 0.5989 && 0.6546 \\ 
 $\varphi_\delta(1,0)$     &&   0.018  && 0.0600  && 0.0847 && 0.1042  && 0.1203 && 0.1391 && 0.1619 && 0.1746 && 0.1953  \\ 
\bottomrule[1.5pt]
\end{tabular}
\end{adjustbox}
\caption{Empirical power comparison of total and partial AUC maximum tests for interaction effects in a two-factor crossed design for different effect sizes.} \label{table:beta_2}
\end{minipage}
\end{table}  

\begin{table}[h!]
\captionsetup{font=footnotesize}
\begin{minipage}{\textwidth}
\begin{adjustbox}{width=0.7\textwidth,center}
\begin{tabular}{c|cccccccccccc}
 $(p\,,\,q)$   &&(1\,,\,0)&&(0.8\,,\,0.2)&&(0.6\,,\,0.4)&&(0.5\,,\,0.5)&&(0.4\,,\,0.6) \\ \midrule[1pt]
p-value        && 0.382  &&  0.259      && 0.069       && 0.015     && 0.051        \\
adj. p-value   && 0.518  &&  0.518      && 0.207       && 0.075     && 0.204       
\end{tabular}
\end{adjustbox}
\caption{Test results for the data example.}\label{table:data}
\end{minipage}
\end{table}
}

\section{Discussion}

\subsection{Accuracy of type I error rate}

{The accuracy of the type I error rate depends on the ability of the bootstrap to approximate our test statistic's distribution, which in turn hinges on the proximity of the test statistic to its weak limit \cite{hall2013}.} Our numerical experiments suggest that this proximity varies between different settings. { We have identified two main factors that we think are responsible for this variation: first, the degree to which the involved distributions $P^{(0)}$ and $P^{(1)}$ overlap; second, the location of the intervals $[F_i^{-1}(1-p),G_i^{-1}(1-q)]$. Let us briefly discuss the influence of these factors.} 

Suppose for the sake of simplicity that $\kappa=1$ and $\alpha_n = \beta_n = n/2$. The derivative $D\tau(F,G)$ is a linear approximation of the functional $\tau$ around $(F,G)$. Hence,
\begin{align*}
&n^{1/2}\big\{ \tau(\mathds{F}_{n},\mathds{G}_{n}) - \tau(F,G) \big\} 
\approx D\tau(F,G)\big[{n}^{1/2}\big((\mathds{F}_{n},\mathds{G}_{n})-(F,G)\big)\big] \\
&= n^{1/2} \int_{(a,b]} \big( \mathds{F}_n(u_j-) - F(u_j-) \big)  \dif G(u) 
 - n^{1/2} \int_{(a,b]} \big( \mathds{G}_n(v_j)\} - G(v_j) \big)   \dif F(u) \\
&\approx \sum_{j=1}^N P^{(1)}\big((u_{j-1},u_j]\big) \ n^{1/2}  \big( \mathds{F}_n(u_j-) - F(u_j-) \big) 
 -  \sum_{j=1}^N P^{(0)}\big((v_{j-1},v_j]\big) \ n^{1/2} \big( \mathds{G}_n(v_j)\} - G(v_j) \big)  \\
&= S_{1,n} + S_{2,n}\,,
\end{align*}
where $a = u_0 < \ldots < u_N = b$ and $a = v_0 < \ldots < v_N = b$ are sufficiently fine partitions and $N$ is sufficiently large. { As above, $a = F^{-1}(1-p)$ and $b = G^{-1}(1-q)$.} By the Lindeberg--Feller theorem and Slutsky's lemma, $S_{1,n}$ and $S_{2,n}$ converge weakly to two mutually independent normally distributed random variables. We invoke the Edgeworth expansion to assess the quality of this normal approximation. Let $\vartheta_j = F(u_j-)$, and denote with $\Tilde{F}_j$ the cumulative distribution function of $n^{-1/2} \sum_{r=1}^n\big( \ind\{\xi_{r} < u_j\} - \vartheta_j \big)\big/(\vartheta_j(1-\vartheta_j))^{1/2}$. Then  
\begin{align}\label{eq:edgeworth}
\Tilde{F}_j(x) = \Phi(x) + \frac{1-2\vartheta_j}{\big(2\pi \; \vartheta_j(1-\vartheta_j) \big)^{1/2}} \big(1-x^2\big) \exp\big(-x^2/2\big) \ n^{-1/2} + o\big(n^{-1/2}\big)\,, 
\end{align}  
as $n \to \infty$ \citep[p. 363]{gut2013}. Of course, the rate $O(n^{-1/2})$ with which $\tilde{F}_j(x)$ converges to $\Phi(x)$ is not affected by the involved model parameters. However, the constants $\vartheta_j$ in the Edgeworth expansions \eqref{eq:edgeworth} { may severely impede the approximation. These constants are sensitive to the shape of the distributions $P^{(0)}$, $P^{(1)}$ as well as the choice of $p$ and $q$.

First, suppose that $a$ and $b$ are very large numbers. This may occur with choices $p <  \varepsilon$ and $q > 1-\varepsilon$ for small $\varepsilon > 0$. The resulting values $F(u_j-) = \vartheta_j$ will be very close to one, implying by virtue of equation \eqref{eq:edgeworth} that 
$F_j(x)$ is not at all close to $\Phi(x)$. Second, if the distributions $P^{(0)}$ and $P^{(1)}$ are well separated with $\mathrm{med}(P^{(1)}) \gg \mathrm{med}(P^{(0)})$, the addends where $F(u_j-)$ is close to one will get a lot of weight within the sums $S_{1,n}$. This again leads to a poor approximation of $S_{1,n}$ by a multivariate normal distribution. Heuristically, both phenomena are plausible if we recall the estimator's shape as a trimmed Mann--Whitney statistic as displayed in line \eqref{eq:trimmed_MW}. A small overlap between $P^{(0)}$ and $P^{(1)}$ imply that only a small number of indicators $\ind\big\{ \xi_{(r)} < \eta_{(s)}\big\}$ will be non-zero. Choices $p <  \varepsilon$ and $q > 1-\varepsilon$ severely reduce the number of addends in \eqref{eq:trimmed_MW} as $\lceil \alpha_n(1-\varepsilon) \rceil +1 \approx \alpha_n$ and $\lceil \beta_n(1-q) \rceil \approx 1$. In both scenarios the effective sample size is drastically reduced. Especially the latter phenomenon might explain the observed variation in the necessary sample sizes in our simulations to attain the aspired type I error rate (cf. Table \ref{table:alpha}).  }

\subsection{Power comparison}

{In our computer experiments, partial AUC maximum tests seemed more successful than the total AUC maximum test in detecting comparable alternatives. We want to give a theoretical explanation for this phenomenon. To this end, we consider two fixed alternatives, $H_1 \colon C^\top \theta_{p,q} = x$ and $H_1' \colon C^\top \theta_{1,0} = x$, for some vector $x \neq 0$. Let us look at the power of the partial AUC maximum test and the total AUC maximum test against the alternatives $H_1$ and $H_1'$, respectively.

First, observe that the power of the partial AUC maximum test against $H_1$ is
\begin{align*}
\ew_{H_1} \, \varphi_\delta(p,q) 
&= \pr_{H_1} \Big\{ \big\| T_n  + (\alpha_n+\beta_n)^{1/2} V_{p,q;n}^{-1/2} \, x \big\|_\infty > 	q(S_n^*;1-\delta) \Big\} \\
&= \pr_{H_1} \Big\{ \big\| \tilde{w}_{p,q}  + (\alpha_n+\beta_n)^{1/2} V_{p,q;n}^{-1/2} \, x \big\|_\infty > 	q(\tilde{w}_{p,q};1-\delta) \Big\} + o(1)\,,
\end{align*}  
where $\tilde{w}$ is a normally distributed random vector with mean 0 and covariance matrix $V_{p,q}^{-1/2}C^\top \Sigma_{p,q} C V_{p,q}^{-1/2}$, and $q(\tilde{w};1-\delta)$ is its $1-\delta$ equicoordinate quantile. Furthermore, $V_{p,q} = \mathrm{diag}(C^\top \Sigma_{p,q} C)$. Analogously, the power of the total AUC maximum test against $H_1'$ is 
\begin{align*}
\ew_{H_1'} \, \varphi_\delta(1,0) =\pr_{H_1'} \Big\{ \big\| \tilde{w}_{1,0}  + (\alpha_n+\beta_n)^{1/2} V_{1,0;n}^{-1/2} \, x \big\|_\infty > 	q(\tilde{w}_{1,0};1-\delta) \Big\} + o(1)\,.
\end{align*} 
We see that for fixed sample size the power of both tests are mainly determined by the size of the elements of $V_{p,q;n}$ and $V_{1,0;n}$, respectively. Moreover, if $V_{p,q;n} < V_{1,0;n}$ elementwise, the partial AUC maximum test will tend to be more powerful against $H_0$ than the total AUC maximum test against $H_0'$. Let us therefore inspect the difference of their asymptotic versions $V_{1,0}(i)-V_{p,q}(i)$, for which we have explicit displays. By definition, 
\begin{align*}
(C^\top \Sigma_{1,0} C)(i,i) - (C^\top \Sigma_{p,q} C)(i,i) = \big \langle c_i\,,\, (\Sigma_{1,0}-\Sigma_{p,q}) \, c_i\big\rangle\,,
\end{align*}
where $c_i$ is the $i$th column of the contrast matrix $C$. A sufficient condition for last display to be positive, regardless of the contrast vector, is that $\Delta := \Sigma_{1,0}-\Sigma_{p,q}$ is positive definite. Since the definiteness of this matrix is hard to ascertain in the general case, we will focus on the special cases where $\kappa = 2$. This corresponds to a comparison of two biomarkers. In this case, $\Delta = \Sigma_{0,1}-\Sigma_{p,q}$ is a symmetric $\Rbb^{2 \times 2}$ matrix with diagonal elements
\begin{align*}
\Delta(i,i) 
&= \int_{\Rbb^2\setminus(a_i,b_i]^2} \ew\big[B_{i}(u)B_{i}(v)\big] \ P^{(1)}_i \otimes P^{(1)}_i(\dif v, \dif u) \\
& \hspace{2cm} +  \int_{\Rbb^2\setminus(a_i,b_i]^2} \ew\big[B_{\kappa + i}(u)B_{\kappa + i}(v)\big] \ P^{(0)}_i \otimes P^{(0)}_i(\dif v, \dif u)\,,  
\end{align*}
which we see by an inspection of the proof of Theorem \ref{thm:convergence2}. Under the condition that $F(x),G(x) < 1$ for all $x< \infty$, we conclude that $\Delta(i,i) > 0$ because  $\ew\big[ B_i(u)B_i(v) \big] = F_i(u\wedge v) - F_i(u)F_i(v) > 0$ and $\ew\big[ B_{\kappa + i}(u)B_{\kappa + i}(v) \big] = G_i(u\wedge v) - G_i(u)G_i(v) > 0$ for every $u,v \in \Rbb$ and $r\in \{1,\ldots,\kappa\}$.
The off-diagonal value of $\Delta$ is 
\begin{align*}
\Delta(1,2) &= \int_{\big((a_1,b_1] \times (a_2,b_2]\big)^c} \ew\big[ B_1(u)B_2(v) \big] \ P^{(1)}_1 \otimes P^{(1)}_2(\dif u, \dif v)  \\
&\hspace{2cm}+ \int_{\big((a_1,b_1] \times (a_2,b_2]\big)^c} \ew\big[ B_1(u)B_2(v) \big] \ P^{(0)}_1 \otimes P^{(0)}_2(\dif u, \dif v)  \,.  									
\end{align*} 
The eigenvalues $\lambda_1 \geq \lambda_2$ of $\Delta$ are the roots of $p(\lambda)=\{\lambda-\Delta(1,1)\}\{\lambda-\Delta(2,2)\} -\{\Delta(1,2)\}^2$. The first eigenvalue, $\lambda_1$, is always positive. The second eigenvalue, $\lambda_2$, is positive if and only if $\Delta(1,1)\,\Delta(2,2) > \{\Delta(1,2)\}^2$. { This will be the case if no strong dependencies between $\sigma\{\eta_{s}(1) \colon s \in \Nbb\}$ and $\sigma\{\eta_{s}(2) \colon s \in \Nbb\}$ or between $\sigma\{\xi_{r}(1) \colon r \in \Nbb_+\}$ and $\sigma\{\xi_{r}(2) \colon r \in \Nbb\}$ are manifest. } Hence, the matrix $\Delta$ is either positive semidefinite or indefinite, depending on the cross covariances of $(B_1,B_2,B_3,B_{4})$. However, it is never negative definite.   

What does this mean heuristically? The partial AUC estimator has asymptotically smaller variances than the total AUC estimator, giving the former an advantage over the latter. However, if $\sigma\{\eta_{s}(1) \colon s \in \Nbb\}$ and $\sigma\{\eta_{s}(2) \colon s \in \Nbb\}$ or $\sigma\{\xi_{r}(1) \colon r \in \Nbb_+\}$ and $\sigma\{\xi_{r}(2) \colon r \in \Nbb\}$ are strongly dependent, we might observe a rather peculiar flow of statistical information between the ROC curves of the two biomarkers: the clinically irrelevant segments of one ROC curve might carry information about the clinically relevant segment of the other ROC curve. This additional information is available only to the total AUC estimator, which might to some degree compensate for the disadvantage imposed by its larger variances. This could result in an indefinite Matrix $\Delta$, corresponding to a situation where for some directions of $x$ the partial AUC maximum test has more power and for other constellations the total AUC maximum test could be more powerful. This would depend on the orientation of $x$ as well as the ellipsoids generated by $\Sigma_{p,q}$ and $\Sigma_{1,0}$. This entire ensemble depends of course on the contrast matrix $C$ and the underlying distributions $P^{(0)}$,
$P^{(1)}$.  

At this point we want to emphasize that the preceding discussion was entirely based on heuristic arguments. A more formal approach would focus on a comparison of the local asymptotic power functions of the total and partial AUC maximum tests. These power functions are derived as the limits of the usual power functions evaluated at local alternatives $x = O(n^{-1/2})$ as $n \to \infty$. To find these limits, our asymptotic results would need to be extended so as to hold uniformly over small sets of marker distributions, which is beyond the scope of this paper.   
}

{

\subsection{Conclusion}

To conclude the paper, we want to discus the consequences of our findings for possible applications. To begin with, let us summarize how our choice of $p$ and $q$ may influence the result of the statistical analysis. First, for a given sample size, the choice of $p$ and $q$ affects the accuracy of the bootstrap approximation of $\hat{\theta}_{p,q}$: we have seen that large values of $F_i^{-1}(1-p)$ tend to increase the sample size that is necessary for a good approximation. Second, for a given alternative $H_1 \colon C^\top \theta_{p,q} = x(p,q) \neq 0$, shrinking the interval $[F_i^{-1}(1-p),G_i^{-1}(1-q)]$ tends to increase the power $\ew_{H_1} \, \varphi_\delta(p,q)$ by reducing the asymptotic variances $V_{p,q}$. In theory, ever smaller intervals $[F_i^{-1}(1-p),G_i^{-1}(1-q)]$ will lead to ever smaller variances $V_{p,q}$. Third, a change in $p$ and $q$ will of course alter the effect size $C^\top \theta_{p,q} = x$. In case of multiple ROC curve crossings, shrinking $[F_i^{-1}(1-p),G_i^{-1}(1-q)]$ may at first increase $x$. Ultimately, however, the effect size will vanish as $\big| F_i^{-1}(1-p)-G_i^{-1}(1-q)\big| \to 0$ because in this case $\theta_{p,q} \to 0$, which means that then also $x(p,q) = C^\top \theta_{p,q} \to 0$. This of course tends to reduce $\ew_{H_1} \, \varphi_\delta(p,q)$. Thus, the choice of $p$ and $q$ influences $\ew_{H_1} \, \varphi_\delta(p,q)$ in different ways that are to some extent contrary to each other. For a given pair of marker distributions $P^{(0)},P^{(1)}$, we could try to choose $p$ and $q$ so as to maximize $\ew_{H_1} \, \varphi_\delta(p,q)$ by balancing out those mechanisms. However, the pair $(p,q)$ yielding the optimal combination of $V_{p,q}$ and $x(p,q)$ certainly depends on the underlying distributions $P^{(0)},P^{(1)}$ and can therefore not be known a priori without additional information. 

In view of the above consideration, we have the following recommendation for the use of $\theta_{p,q}$ as a statistical parameter, depending on the study type. For exploratory analysis in a retrospective study, we would conduct inference for several choices of $p$ and $q$. This may prevent candidate biomarkers from being thrown out of the research process too early just because they don't seem superior to the current diagnostic standard in terms of their total AUCs. Instead we might uncover statistically significant differences for some clinically relevant ROC curve segments. This approach leads of course to a more general multiple testing problem based on the family of null hypotheses $\big\{H_0^{i,j,k} \colon \langle c_i \,,\, \theta_{p_j,q_k} \rangle = 0 \text{ with } i=1,\ldots,r$; $j=1,\ldots,s$; $k=1,\ldots, t \big\}$. We think that our theoretical approach is suitable for constructing a test procedure for this problem as well. The main ingredient would be an extension of our asymptotic results so as to cover the asymptotic distribution of vectors $\big( \hat{\theta}_{p_1,q_1}, \ldots, \hat{\theta}_{p_s,q_t} \big)^\top$. Our empirical process approach is certainly suitable for this problem, but a detailed derivation seems to be out of scope of the present paper.

By contrast, a confirmatory prospective study certainly demands a prior specification of the parameter of interest. With certain information from prior exploratory studies available, a proper sample size planning should be possible. In this case a choice directed at maximizing $\ew_{H_1} \, \varphi_\delta(p,q)$ may enable a more economical decision regarding the sample sizes. Lastly,   
we have to emphasize that the most important criterion for a good choice of $p$ and $q$ is a justification why the corresponding parameter $\theta_{p,q}$ is really appropriate to answer the research question. In our context, this means explaining why we want to assess a biomarker's performance based solely on the ROC curve segment defined by $p$ and $q$. This in turn is a question not so much for the statistician but for the medical investigators of the study.
}

\section*{Acknowledgements}
The authors gratefully acknowledge financial support from the Deutsche Forschungsgemeinschaft (grant nos. DFG KO 4680/3-2, DFG KO 4680/4-1).
We are grateful to the Charité Department of Radiology for providing the example data set. We thank two anonymous reviewers for their valuable comments which led to considerable improvements in the manuscript.

\bibliographystyle{myjmva}
\bibliography{references}

\newpage

\appendix

\section{Proofs}\label{sec:proofs}

For $t \in \Rbb^{2\kappa}$ let $\pi_t \colon (\ell_\infty^{2\kappa},d^{2\kappa}_\infty) \to \Rbb$ be the coordinate projection $\pi_t(x) = x(t)$. With $\cp^{2\kappa}$ we denote the projection 
$\sigma$-field or cylindrical $\sigma$-field on $(\ell_\infty^{2\kappa},d^{2\kappa}_\infty)$, i.e. the $\sigma$-field generated by the coordinate projections. In the following, the symbol $\rightsquigarrow$ denotes weak convergence.  

\begin{proof}[Proof of Theorem \ref{thm:convergence1}]
We extend the result of Theorem \ref{thm:k=1} of the supplemental material, combining the ideas of \citet[p. 42]{vandervaart1996} and \citet[p. 97]{Pollard1984}. Essentially we use the continuous mapping theorem twice. 

First, the map $\varphi \colon (\ell_\infty^{2\kappa}([0,1)),d_\infty^{2\kappa}) \to (\ell_\infty([0,{2\kappa})),d_\infty)$,
\begin{align*}
\varphi(x_1,\ldots,x_{2\kappa})[t] = \sum_{i=1}^{2\kappa} x_i(t-i+1) \, \ind_{[i-1,i)}(t)\,, \qquad t \in [0,{2\kappa})\,,
\end{align*}
is an isometry. Hence, both $\varphi$ and $\varphi^{-1}$ are continuous, and with the continuous mapping theorem \citep[Theorem 1.3.6]{vandervaart1996} we conclude that $\varphi(x_{1,n},\ldots,x_{{2\kappa},n}) \rightsquigarrow \varphi(x_{1,0},\ldots,x_{{2\kappa},0})$ in $\ell_\infty([0,{2\kappa}))$ if and only if $(x_{1,n},\ldots,x_{{2\kappa},n}) \rightsquigarrow (x_{1,0},\ldots,x_{{2\kappa},0})$ in $\ell_\infty^{2\kappa}([0,1))$, for any sequence $\big\{(x_{1,n},\ldots,x_{{2\kappa},n}) \colon n \in \Nbb \big\} \subset (\ell_\infty^{2\kappa}([0,1)),d_\infty^{2\kappa})$. 

Now for $r \in \Nbb_+$, let $(\upsilon_{1,r},\ldots,\upsilon_{{2\kappa},r}) \colon (\ell_\infty^{2\kappa},\cp^{2\kappa}) \to \big([0,1]^{2\kappa},\cb([0,1]^{2\kappa})\big)$ be a $[0,1]^{2\kappa}$-valued random vector whose marginals are uniformly distributed over $[0,1]$. According to Sklar's Theorem \citep{Nelsen2006}, its cumulative distribution function $C$ can be chosen such that $H(t_1,\ldots,t_{2\kappa}) = C\big(F_1(t_1),\ldots,F_\kappa(t_\kappa),G_{1}(t_{\kappa+1}),\ldots,G_\kappa(t_{2\kappa})\big)$. Therefore, $\pr \big\{F_1^{-1}(\upsilon_{1,r}) \leq t_1 ,\ldots, G_{\kappa}^{-1}(\upsilon_{{2\kappa},r}) \leq t_{2\kappa} \big\} = H(t_1,\ldots,t_{2\kappa})$, and we conclude that $\big(F_1^{-1}(\upsilon_{1,r}),\ldots, G_{\kappa}^{-1}(\upsilon_{{2\kappa},r})\big)$ and $X_n$ have the same distribution on $(\ell_\infty^{2\kappa},\cp^{2\kappa})$.

Now consider $U_n = \big(U_{1,n},\ldots,U_{{2\kappa},n}\big) \in (\ell_\infty^{2\kappa}([0,1)),d_\infty^{2\kappa})$, 
\begin{align*}
U_{i,n}(t) :=  n^{-1/2}\sum_{r=1}^n \big(\ind_{\{\upsilon_{i,r} \leq t\}} - \tilde{G}(t)\big)\,,  \hspace{1cm} i \in \{1,\ldots,{2\kappa}\}\,,
\end{align*}
where $\tilde{G}$ denotes the cumulative distribution function of a $\text{Uniform}[0,1]$ distribution. For all $i \in \{1,\ldots,2\kappa\}$, Donsker's Theorem \citep[p. 96]{Pollard1984} implies that $U_{i,n} \rightsquigarrow U_{i,0}$ in $(\ell_\infty[0,1),d_\infty)$, where $U_{i,0}$ is a $\tilde{G}$-brownian bridge. By Prohorov's Theorem \citep[Thm 1.3.9]{vandervaart1996} in combination with Lemmas 1.4.3 and 1.4.4 of \cite{vandervaart1996}, we conclude that there exists a subsequence $\{U_{n_q}\}$ that converges weakly to a tight Borel measurable random element $U_0 \in (\ell_\infty^{2\kappa}[0,1),d_\infty^{2\kappa})$. Suppose there is another such subsequence $\{U_{n_r}\}$ with corresponding weak limit $V_0$. As noted before, $\varphi(U_{n_q}) \rightsquigarrow \varphi(U_0)$ and $\varphi(U_{n_r}) \rightsquigarrow \varphi(V_0)$. Both limits are Borel measurable tight random elements. We will show that $\varphi(V_0)$ and $\varphi(U_0)$ have the same distribution. For any $k \in \Nbb$, the mapping $x \mapsto (\pi_{t_1}x,\ldots,\pi_{t_k}x)\colon (\ell_\infty([0,{2\kappa})),d_\infty) \to \Rbb^k$ is continuous. By the continuous mapping theorem, $\big(\pi_{t_1}\varphi(U_{n_q}), \ldots,\pi_{t_k}\varphi(U_{n_q})\big) \rightsquigarrow \big(\pi_{t_1}\varphi(U_{0}), \ldots,\pi_{t_k}\varphi(U_{0})\big)$ as well as 
$\big(\pi_{t_1}\varphi(U_{n_r}), \ldots,\pi_{t_k}\varphi(U_{n_r})\big) \rightsquigarrow \big(\pi_{t_1}\varphi(V_{0}), \ldots,\pi_{t_k}\varphi(V_{0})\big)$. But by the multivariate Lindeberg-Feller Theorem, both limit distributions coincide. Since tight Borel probability measures on $(\ell_\infty,\cb(\ell_\infty))$ are uniquely defined by their finite dimensional marginals \citep[Lemma 1.5.3]{vandervaart1996}, we conclude that $\varphi(U_0)$ and $\varphi(V_0)$ have the same distribution. Since all subsequences of $\varphi(U_n)$ converge to the same weak limit $\varphi(U_0)$, it follows that $\varphi(U_n) \rightsquigarrow \varphi(U_0)$, which implies $U_n \rightsquigarrow U_0$. Note that $\phi(U_0)$ is a centered Gaussian process and so is $U_0$. 

With $F_i(-\infty) = 0$ and $F(\infty) = 1$, we can view each $F_i$ as an element in $(\ell_\infty,d_\infty)$.  We define the map $\tilde{H} \colon (\ell_\infty^{2\kappa}([0,1)),d^{2\kappa}_\infty) \to (\ell_\infty^{2\kappa},d^{2\kappa}_\infty)$ by $\tilde{H} \big( x_1,\ldots, x_{2\kappa} ) := \big(x_1(F_1(\cdot)), \ldots, x_{2\kappa}( G_{\kappa}(\cdot))\big)$. Observe that $d^{2\kappa}_\infty \big(\tilde{H}(x),\tilde{H}(y)\big) \leq d^{2\kappa}_\infty(x,y)$ for all $x,y \in (\ell_\infty^{2\kappa}([0,1)),d^{2\kappa}_\infty)$, which makes $\tilde{H}$ continuous. Therefore,  
\begin{align*}
\Big(
n^{-1/2}\sum_{r=1}^n \big(\ind_{\{F_1^{-1}(\upsilon_{1,r}) \leq t\}} - F_1(t)\big), \ldots,
n^{-1/2}\sum_{r=1}^n \big(\ind_{\{G_{\kappa}^{-1}(\upsilon_{{2\kappa},r}) \leq t\}} - G_{\kappa}(t)\big)
\Big)
= \tilde{H}(U_n) \rightsquigarrow \tilde{H}(U_0)\,. 
\end{align*}
However, for any $i \in \{1,\ldots,\kappa\}$, $F_i$ and $G_i$ are the cumulative distribution functions of $F_i^{-1}(\upsilon_{i,r})$ and $G_i^{-1}(\upsilon_{i,r})$, respectively, which means that $\tilde{H}(U_n)$ and $X_{n}$ have the same distribution on $(\ell_\infty^{2\kappa},\cp^{2\kappa})$. With Lemma \ref{lemma:1} from the supplemental material we conclude that $X_{n} \rightsquigarrow \tilde{H}(U_0)$ as well. From the definition of $\tilde{H}$, it follows that $\tilde{H}(U_{0})$ is a centered Gaussian process. 

In a lest step, we compute the corresponding covariance function. We start with the covariance of $U_0 = (U_{1,0}, \ldots, U_{{2\kappa},0})$. For $0 \leq s \leq t \leq 1$ and $i,j \in \{1,\ldots,{2\kappa}\}$, 
\begin{align*}
\cov \big( U_{i,0}(s),U_{j,0}(t) \big) 
= \ew \Big[ \ind_{\{\upsilon_{i,1} \leq s\}} \ind_{\{\upsilon_{j,1} \leq t\}} - \tilde{G}(s)\tilde{G}(t) \Big] 
= C\big(k_{i,j}(s,t)\big) - \tilde{G}(s)\tilde{G}(t)\,.
\end{align*}
Therefore, for $-\infty < s \leq t < \infty$ and $i,j \in \{1, \ldots, {\kappa}\}$,
\begin{align*}
\cov \big( (\tilde{H}(U_0))_i[s],(\tilde{H}(U_0))_j[t] \big)
&= \cov \big( U_{i,0}(F_i(s)),U_{j,0}(F_j(t)) \big) \\
&= C \big( k_{i,j}(F_i(s),F_j(t)) \big) - \tilde{G}(F_i(s))\, \tilde{G}(F_j(t)) \\
&= F \big( k_{i,j}(s,t) \big) - F_i(s)\, F_j(t) \,.
\end{align*}
For the cases $i,j \in \{\kappa +1,\ldots, 2\kappa\}$ as well as $i \in \{1,\ldots,\kappa\}$, $j \in \{\kappa +1,\ldots,2\kappa\}$, the argument is completely analogous.
\end{proof}

\begin{proof}[Proof of Lemma \ref{lemma:hadamard_dif}]
First we show that $\psi_{p,q}^{(i)}$ is Hadamard differentiable at $(F_1,\ldots,F_{\kappa},G_1,\ldots,G_\kappa)$ tangentially to $\mathds{D}_0$, with derivative
\begin{align*}
D\psi_{p,q}^{(i)}(F_1,\ldots,F_{\kappa},G_{1}, \ldots, G_\kappa)[h_1,\ldots,h_{2\kappa}] = \Bigg(h_{i},h_{\kappa + i}, -\frac{h_{i}(F_{i}^{-1}(1-p)}{f_{i}(F_{i}^{-1}(1-p)},-\frac{h_{\kappa + i}\big(G_{i}^{-1}(1-q)\big)}{g_{i}\big(G_{i}^{-1}(1-q)\big)}\Bigg)\,.
\end{align*}
Let the sequences $\{t_r\}_{r \in \Nbb} \subset \Rbb$ and $\{(h_{1,r}, \ldots,h_{2\kappa,r})\}_{r \in \Nbb} \subset \mathds{D}_1$ be chosen such that $t_r \to 0$ and $\|(h_{1,r},\ldots,h_{2\kappa,r}) - (h_1,\ldots,h_{2\kappa})\|_{\mathds{D}_1} \to 0$ as $r \to \infty$. Note that
\begin{align*}
&\frac{1}{t_r} \big( \psi_{p,q}^{(i)}(F_1+t_rh_{1,r},\ldots,G_{\kappa}+t_rh_{2\kappa,k}\big) - \psi_{p,q}^{(i)}(F_1,\ldots,G_{\kappa}) \big) \\
&= \Big( h_{i,r} , h_{i+\kappa,r} , \frac{(F_{i}+t_r h_{i,r})^{-1}(1-q)-F_{i}^{-1}(1-q)}{t_r},\frac{(G_{i}+t_r h_{\kappa  + i,r})^{-1}(1-p)-G_{i}^{-1}(1-p)}{t_r} \Big)\,.
\end{align*}
This converges in $\|\cdot\|_{\mathds{D}_2}$ to the prescribed limit (consult Lemma 21.3 in \cite{vandervaart2000} for the Hadamard derivative of the quantile transform).

In a second step, we have to show that $\phi_{p,q}$ is Hadard differentiable at each  $(F_{i},G_{i})$, tangentially to $D\psi_{p,q}^{(i)}(F_1,\ldots,G_{\kappa})[\mathds{D}_0]$, with derivative
\begin{align*}
&D\phi_{p,q}(F_{i},G_{i},a,b)[h_{i},h_{\kappa + i},x,y]  \\
&\hspace{2.2cm}= \big[ y \, g_{i}(b) \big( F_{i}(b) - F_{i}(a) \big) - x \, f_{i}(a) \big( G_{i}(b) - G_{i}(a)\big) \big]
\\
&\hspace{3.5cm} + \int_{(a,b]} \big(h_{\kappa + i}(b) - h_{\kappa + i}(u)\big) \ \dif F_{i}(u) 
  + \int_{(a,b]} \big(h_{i}(u-) - h_{i}(a)\big) \ \dif G_{i}(u) \\
&\hspace{2.2cm}=  \int_{(a,b]} \big(h_{\kappa + i}(b) - h_{\kappa + i}(u) + y \, g_{i}(b) \big) \ \dif F_{i}(u) 
  + \int_{(a,b]} \big(h_{i}(u-) - h_{i}(a) - x \, f_{i}(a) \big) \ \dif G_{i}(u)
\end{align*}
To that end, let $\{(h_{{i},r},h_{{\kappa + i},r},x_r,y_r)\}_{r \in \Nbb}$ be a sequence such that $\|(h_{{i},r},h_{{\kappa + i},r},x_r,y_r) - (h_{i},h_{\kappa + i},x,y)\|_{\mathds{D}_2} \to 0$, as $r \to \infty$, and $\big(F_{i} + t_r h_{i,r}, G_{i} + t_rh_{\kappa + i,r}, a + t_r x_r, b + t_r y_r\big) \in \mathds{D}_2$ for all $r$. Then
\begin{align*}
&\frac{1}{t_r} \big[ \phi_{p,q}(F_{i} + t_rh_{{i},r},G_{i} + t_rh_{{\kappa + i},r}, a + t_r x_r, b + t_r y_r) - \phi_{p,q}(F_{i},G_{i},a,b) \big] \\ 
&\hspace{0.5cm}= \frac{1}{t_r} \big[ \phi_{p,q}(F_{i} + t_r h_{{i},r},G_{i} + t_r h_{{\kappa + i},r}, a + t_r x_r, b + t_r x_r) \\
&\hspace{1.5cm} - \phi_{p,q}(F_{i} + t_r h_{{i},r},G_{i} + t_r h_{{\kappa + i},r},a,b) \big] \\
& \hspace{3cm} + \frac{1}{t_r} \big[\phi_{p,q}(F_{i} + t_r h_{{i},r},G_{i} + t_r h_{{\kappa + i},r},a,b) - \phi_{p,q}(F_{i},G_{i},a,b) \big] \\
& \hspace{.5cm}= D^{(1)}_{i,r} + D^{(2)}_{i,r}\,.
\end{align*}
It can be shown that
\begin{align}
&\Big| D^{(1)}_{i,r}  - \Big[y \, g_{i}(b) \big( F_{i}(b) - F_{i}(a) \big) 
- x \, f_{i}(a) \big( G_{i}(b) - G_{i}(a)\big) \Big] \Big| \nonumber \\
&\hspace{0.5cm} + \Big| D^{(2)}_{i,r} -  \Big[\int_{(a,b]} \big(h_{\kappa + i}(b) - h_{\kappa + i}(u)\big) \ \dif F_{i}(u) 
  +  \int_{(a,b]} \big(h_{i}(u-) - h_{i}(a)\big) \ \dif G_{i}(u) \Big] \Big| \to 0\,,
  \label{eq:D1+D2:2}
\end{align}
as $r \to \infty$. We present a detailed proof of this claim in the supplemental material.

According to the chain rule for Hadamard differentiation \citep[Lemma 3.9.6]{vandervaart1996},
\begin{align*}
&D \big( \phi_{p,q} \circ \psi_{p,q}^{(i)} \big)(F_1,\ldots,G_{\kappa})[h_1, \ldots,h_{2\kappa}] \\
&= D \phi_{p,q}\big(\psi_{p,q}^{(i)}(F_1,\ldots,G_{\kappa})\big)\big[ D\psi_{p,q}^{(i)}(F_1,\ldots,G_{\kappa})[h_1,\ldots,h_{2\kappa}]\big] \\
&= \int_{\big(F_{i}^{-1}(1-p),G_{i}^{-1}(1-q)\big]} h_{i}(u-) \ \dif G_{i}(u) - \int_{\big(F_{i}^{-1}(1-p),G_{i}^{-1}(1-q)\big]} h_{\kappa + i}(u)  \ \dif F_{i}(u) \,.
\end{align*}
This implies the first assertion. 

As to the derivative of $\tau_{1,0}$, note that we do not have to apply the chain rule and that the term $D_{i,k}^{(1)}$ disappears in this case. The analysis of $D^{(2)}_{i,k}$ is very similar to the case of $\tau_{p,q}$. 
\end{proof}

\begin{proof}[Proof of Theorem \ref{thm:convergence2}]
The weak convergence to $w_{p,q}$ follows from the functional delta method \citep[Theorem 3.9.4]{vandervaart1996}. As linear functionals of Gaussian processes, the limits are Gaussian, i.e. multivariate normal \citep[Lemma 3.9.8]{vandervaart1996}. The covariance $\cov\big(\langle w_{p,q}, e_i \rangle\,,\,\langle w_{p,q}, e_j \rangle\big)$ is a sum of four covariances, 
\begin{align*}
&\cov \Bigg( 
\int_{(a_i,b_i]} \lambda^{-1/2} B_{i}(u-) \ \dif G_{i}(u) 
- \int_{(a_i,b_i]} (1-\lambda)^{-1/2} B_{\kappa + i}(u) \  \dif F_{i}(u) \,,\, \\
& \hspace{4cm} \int_{(a_j,b_j]} \lambda^{-1/2} B_{j}(u-) \ \dif G_{j}(u) 
- \int_{(a_j,b_j]} (1-\lambda)^{-1/2} B_{\kappa + j}(u) \  \dif F_{j}(u)
 \Bigg)  \\
&= T_1 + T_2 + T_3 + T_4\,. 
\end{align*}
Using the fact that the brownian bridge has continuous sample paths and that $\ew B_i(u) = 0$ for any $u \in \Rbb$, we can apply Fubini's theorem to conclude that
\begin{align*}
T_1 
&= \int_{(a_i,b_i]} \int_{(a_j,b_j]} \lambda^{-1}
\ew   \big[ B_{i}(u) B_{j}(v-) \big] \ P^{(1)}_{j}(\dif v)  P^{(1)}_{i}(\dif u)  = \lambda^{-1} \iint_D \cov\big(B_{i}(u),B_j(v)\big) \ P^{(1)}_{i} \otimes P^{(1)}_{j}(\dif u,\dif v)\,.
\end{align*}
The other three terms are computed analogously. 

The proof that $\hat{\sigma}_{p,q;n}(i,j) \to \sigma_{p,q}(i,j)$ almost surely as $n \to \infty$ is entirely based on the Glivenko--Cantelli theorem. We provide it in the supplemental material.
\end{proof}

\begin{proof}[Proof of Theorem \ref{thm:boostrap}]
The claim follows upon an application of the fundamental Theorem \ref{Thm:boostrap_fundamental} of the supplemental material \citep[Theorem 23.9]{vandervaart2000}. To verify its conditions, we use our Theorem \ref{thm:convergence1} and Theorem 3.6.1 of \cite{vandervaart1996}. 
\end{proof}

\begin{proof}[Proof of Lemma \ref{lemma_test}]
Let $V$ be the diagonal matrix with $\mathrm{diag}(V)=\mathrm{diag}(C^\top \Sigma_{p,q} C)$. As a consequence of Theorem \ref{thm:boostrap}, Lemma \ref{lemma:boostrap_var}, and Slutsky's lemma, we have
\begin{align*}
\sup_{x \in \Rbb} \big| \pr \big\{\|S_n^*\|_\infty \leq x  \bigm| \xi_1,\ldots,\xi_{\alpha_n},\eta_1,\ldots,\eta_{_{\beta_n}} \big\} - \pr \big\{\|V^{-1/2}C^\top w_{p,q}\|_\infty \leq x \big\} \big| \to 0
\end{align*}
as $n \to \infty$, in probability \citep[p. 862]{gine1990}. Hence, for all $\theta_{p,q}$ such that $\langle c_{i_1}\,,\,\theta_{p,q})\rangle = \ldots = \langle c_{i_m}\,,\,\theta_{p,q})\rangle = 0$,
\begin{align*}
 \pr \Big( \bigcup_{j=1}^m \big\{ \varphi_{i_j, \delta} = 1 \big\} \Big) 
&= \pr \Big\{ \max_{j \in \{i_1,\ldots,i_m\}} \big|T_{j,n}\big| > q_{1,\ldots,r}(S_n^*;1-\delta) \Big\} \\
&\leq \pr \Big\{ \max_{j \in \{i_1,\ldots,i_m\}} \big|T_{j,n}\big| > q_{i_1,\ldots,i_m}(S_n^*;1-\delta) \Big\} \\
&= \pr \Big\{ \max_{j \in \{i_1,\ldots,i_m\}} \big|n^{1/2}\big(\big\langle c_j ,\hat{\theta}_{p,q}[\mathds{P}_n]\big\rangle - \big\langle c_j , \theta_{p,q}[P]\big\rangle \big) \big| > q_{i_1,\ldots,i_m}(S_n^*;1-\delta)\Big\} \to \delta
\end{align*}
as $n \to \infty$, where $q_{i_1,\ldots,i_m}(S_n^*;1-\delta)$ denotes the $(1-\delta)$ the equi-coordinate quantile of the random vector $\big(\langle S_n^*,e_{i_1}\rangle,\ldots,\langle S_n^*,e_{i_m}\rangle\big)^\top$. For the global test, i.e. if $\{i_1,\ldots,i_m\} = \{1,\ldots,r\}$, the inequality in the second line becomes an equality. 
\end{proof}

\section{Supplemental Material}

\subsection*{Details of the proof of Lemma \ref{lemma:hadamard_dif}}

For the Lebesgue--Stieltjes measure of a càdlàg function $F$, we write $\mu_F$.
Without loss of generality we consider the case $i=1$. Let us first estimate the size of $D_{1,k}^{(1)}$. We have
\begin{align*}
t_r \, D_{1,r}^{(1)} 
&= \int_{(a+x t_r,b+y t_r]} \big(G_1 + t_r h_{\kappa+1,r}\big)(b+y t_r) - \big(G_1 + t_r h_{\kappa + 1,r}\big)(u) \ \dif \big(F_1 + t_r h_{1,r}\big)(u) \\
& \hspace{2cm} - \int_{(a,b]} \big(G_1 + t_r h_{\kappa + 1,r}\big)(b) - \big(G_1 + t_r h_{\kappa+1,r}\big)(u) \ \dif \big(F_1 + t_r h_{1,r}\big)(u) \\
&= \int_{(-\infty,b+y t_r]} \big(G_1 + t_r h_{\kappa + 1,r}\big)(b+y t_r) - \big(G_1 + t_r h_{\kappa + 1,r}\big)(u) \ \dif \big(F_1 + t_r h_{1,r}\big)(u) \\
& \hspace{.5cm} - \int_{(-\infty,a+x t_r]} \big(G_1 + t_r h_{\kappa + 1,r}\big)(b+y t_r) - \big(G_1 + t_r h_{\kappa + 1,r}\big)(u) \ \dif \big(F_1 + t_r h_{1,r}\big)(u) \\
& \hspace{1.5cm} - \int_{(-\infty,b]} \big(G_1 + t_r h_{\kappa + 1,r}\big)(b) - \big(G_1 + t_r h_{\kappa + 1,r}\big)(u) \ \dif \big(F_1 + t_r h_{1,r}\big)(u) \\
& \hspace{2cm} + \int_{(-\infty,a]} \big(G_1 + t_r h_{\kappa + 1,r}\big)(b) - \big(G_1 + t_r h_{\kappa + 1,r}\big)(u) \ \dif \big(F_1 + t_r h_{1,r}\big)(u) \\
&= \qquad \int_{(-\infty,b+y t_r]} \big(G_1 + t_r h_{\kappa + 1,r}\big)(b+y t_r) - \big(G_1 + t_r h_{\kappa + 1,r}\big)(b) \ \dif \big(F_1 + t_r h_{1,r}\big)(u) \\
&\qquad      + \int_{(-\infty,b+y t_r]} \big(G_1 + t_r h_{\kappa + 1,r}\big)(b) - \big(G_1 + t_r h_{\kappa 1,r}\big)(u) \ \dif \big(F_1 + t_r h_{1,r}\big)(u) \\
&\qquad - \int_{(-\infty,b]} \big(G_1 + t_r h_{\kappa + 1,r}\big)(b) - \big(G_1 + t_r h_{\kappa + 1,r}\big)(u) \ \dif \big(F_1 + t_r h_{1,r}\big)(u) \\
&\qquad - \int_{(-\infty,a+x t_r]} \big(G_1 + t_r h_{\kappa + 1,r}\big)(b+y t_r) - \big(G_1 + t_r h_{\kappa + 1,r}\big)(b) \ \dif \big(F_1 + t_r h_{1,r}\big)(u) \\
&\qquad - \int_{(-\infty,a+x t_r]} \big(G_1 + t_r h_{\kappa + 1,r}\big)(b) - \big(G_1 + t_r h_{\kappa + 1,r}\big)(u) \ \dif \big(F_1 + t_r h_{1,r}\big)(u) \\
&\qquad + \int_{(-\infty,a]} \big(G_1 + t_r h_{\kappa + 1,r}\big)(b) - \big(G_1 + t_r h_{\kappa + 1,r}\big)(u) \ \dif \big(F_1 + t_r h_{1,r}\big)(u) \,.
\end{align*}
As $r \to \infty$, we see that
\begin{align*}
& \frac{1}{y t_r} \int_{(-\infty,b+y t_r]} \big(G_1 + t_r h_{\kappa + 1,r}\big)(b+y t_r) - \big(G_1 + t_r h_{\kappa + 1,r}\big)(b) \ \dif \big(F_1 + t_r h_{1,r}\big)(u) \\
& = \Big[\frac{G_1(b + y t_r) - G_1(b)}{y t_r} + \frac{h_{\kappa + 1,r}(b + y t_r) - h_{\kappa + 1,r}(b)}{y}\Big] \Big[ F_1(b + t_r y) + t_r \mu_{h_{1,r}}\big( (-\infty,b+t_r y]\big) \Big] \\
&= \big(g_1(b) + o(1) \big) \, \big( F_1(b) + o(1) + O(t_r) \big) \,,
\end{align*}
and analogously
\begin{align*}
&\frac{1}{y t_r} \int_{(-\infty,a+x t_r]} \big(G_1 + t_r h_{\kappa + 1,r}\big)(b+y t_r) - \big(G_1 + t_r h_{\kappa + 1,r}\big)(b) \ \dif \big(F_1 + t_r h_{1,r}\big)(u) \\
&= \big(g_1(b) + o(1) \big) \, \big(F_1(a) + o(1) + O(t_r) \big) \,.
\end{align*}
If $r$ is large enough, both $F_1$ and $G_1$ are continuously differentiable in $(b-|y t_r|,b+|y t_r|)$. Therefore, 
\begin{align*}
&\frac{1}{y\,t_r} \int_{(-\infty,b+y t_r]} \big( G_1(b)-G_1(u) \big) \ \dif F_1(u) 
- \frac{1}{y\,t_r} \int_{(-\infty,b]} \big( G_1(b)-G_1(u) \big) \ \dif F_1(u) \\
&=   \frac{1}{|y\,t_r|} \int_b^{b+y t_r} \big( G_1(b)-G_1(u) \big) \, f_1(u) \ \dif u \\
&= \big( G_1(b)-G_1(z_r) \big) \, f_1(z_r)\,,
\end{align*}
for some $z_r \in \Rbb$ with $|z_r - b| \leq |y \, t_r|$. Hence,
\begin{align*}
&\frac{1}{y\,t_r} \Big| \int_{(-\infty,b+y t_r]} \Big[ \big(G_1 + t_r h_{\kappa + 1,r}\big)(b) - \big(G_1 + t_r h_{\kappa + 1,r}\big)(u)\Big] \ \dif \big(F_1 + t_r h_{1,r}\big)(u) \\
&\hspace{2cm} - \int_{(-\infty,b]} \big(G_1 + t_r h_{\kappa + 1,r}\big)(b) - \big(G_1 + t_r h_{\kappa + 1,r}\big)(u) \ \dif \big(F_1 + t_r h_{1,r}\big)(u) \Big| \\
&\leq \frac{1}{y\,t_r} \Big| \int_{(-\infty,b+t_r]} \big( G_1(b)-G_1(u) \big) \ \dif F_1(u) 
- \frac{1}{y\,t_r} \int_{(-\infty,b]} \big( G_1(b)-G_1(u) \big) \ \dif F_1(u) \Big|\\
& \hspace{1.5cm} + \frac{2\,\sup_r \|h_{\kappa + 1,r}\|_\infty}{y}   \mu_{F_1}\big( (b,b+y t_r] \cup (b+y t_r,b] \big) 
 + \frac{2 + o(1)}{y}   \mu_{h_{1,r}}\big( (b,b+y t_r] \cup (b+y t_r,b] \big)\\
&= |G_1(b) - G_1(z_k)| \, |f_1(z_k)| + o(1)\\
&= o(1)
\end{align*}
for $k \to \infty$. Analogously we infer that
\begin{align*}
&\frac{1}{x\,t_r} \Big[ \int_{(-\infty,a+x t_r]} \Big[ \big(G_1 + t_r h_{\kappa + 1,r}\big)(b) - \big(G_1 + t_r h_{\kappa + 1,r}\big)(u)\Big] \ \dif \big(F_1 + t_r h_{1,r}\big)(u) \\
&\hspace{3cm} - \int_{(-\infty,a]} \big(G_1 + t_r h_{\kappa + 1,r}\big)(b) - \big(G_1 + t_r h_{\kappa + 1,r}\big)(u) \ \dif \big(F_1 + t_r h_{1,r}\big)(u) \Big] \\
&= (G_1(b) - G_1(a)) \, f_1(a) + o(1)  \,,
\end{align*}
as $r \to \infty$. In summary, we have shown so far that
\begin{align*}
\lim_{r \to \infty} \Big|D^{(1)}_{1,r} - \big[ y \, g_1(b) \big( F_1(b) - F_1(a) \big) - x \, f_1(a) \big( G_1(b) - G_1(a)\big) \big] \Big| = 0\,.
\end{align*}
As for $D^{(2)}_{1,k}$, we see that
\begin{align*}
D^{(2)}_{1,r} 
&= \frac{1}{t_r} \Big[ \int_{(a,b]} \big[ \big(G_1 + t_r h_{\kappa + 1,r}\big)(b) - \big(G_1 + t_r	h_{\kappa + 1,r}\big)(u) \big] \ \dif \big(F_1 + t_r h_{1,r}\big)(u) \\
&\hspace{4cm} - \int_{(a,b]} \big( G_1(b) - G_1(u) \big) \ \dif F_1(u) \Big] \\
&= \int_{(a,b]} \big(h_{\kappa + 1,r}(b) - h_{\kappa + 1,r}(u)\big) \ \dif F_1(u) \\
&\hspace{2cm} + \int_{(a,b]} \big( G_1(b) - G_1(u) \big) \ \dif h_{1,r}(u) \\
&\hspace{4cm} + t_r \int_{(a,b]} \big(h_{\kappa + 1,r}(b) - h_{\kappa + 1,r}(u) \big) \ \dif h_{1,r}(u) \\
&= \int_{(a,b]} \big(h_{\kappa + 1,r}(b) - h_{\kappa + 1,r}(u)\big) \ \dif F_1(u) \\
& \hspace{2cm} + \int_{(a,b]} \big(h_{1,r}(u-) - h_{1,r}(a)\big) \ \dif G_1(u) \\
& \hspace{4cm} + O(t_r) \\
&= \int_{(a,b]} \big(h_{\kappa + 1}(b) - h_{\kappa + 1}(u)\big) \ \dif F_1(u)  + \int_{(a,b]} \big(h_{1}(u-) - h_{1}(a)\big) \ \dif G_1(u) 
 + o(1)\,,
\end{align*}
since $\|h_{\kappa + 1,r}-h_{\kappa + 1}\|_\infty + \|h_{1,r} - h_1\|_\infty \to 0$, as $r \to \infty$.
\qed

\subsection*{Proof of Corollary 1}
Let $\mathds{P}^{(0)}_{i,n} = \alpha_n^{-1} \sum_{r=1}^{\alpha_n} \delta_{\xi_r}$ and $\mathds{P}^{(1)}_{i,n} = \beta_n^{-1} \sum_{s=1}^{\beta_n} \delta_{\eta_s}$.
Recall that
\begin{align}\label{eq:eq}
\hat{\sigma}_{p,q;n}(i,j)&= \frac{\alpha_n + \beta_n}{\alpha_n}\int_{(\hat{a}_{i,n},\hat{b}_{i,n}]} \int_{(\hat{a}_{j,n},\hat{b}_{j,n}]} \big( \mathds{F}_n\big( k_{i,j}(u,v)\big) - \mathds{F}_{i,n}(u)\mathds{F}_{j,n}(v) \big) \ \dif \mathds{G}_{j,n}(v)  \dif \mathds{G}_{i,n}(u) \nonumber \\ 
&\qquad+ \frac{\alpha_n + \beta_n}{\beta_n}\int_{(\hat{a}_{i,n},\hat{b}_{i,n}]} \int_{(\hat{a}_{j,n},\hat{b}_{j,n}]}  \big( \mathds{G}_n\big( k_{i,j}(u,v)\big) - \mathds{G}_{i,n}(u)\mathds{G}_{j}(v) \big) \ \dif \mathds{F}_{j,n}(v)  \dif \mathds{F}_{i,n}(u) \,.
\end{align}
We consider both addends separately. As for the first term, observe that
\begin{align*}
&\Big|\frac{\alpha_n + \beta_n}{\alpha_n}\int_{(\hat{a}_{i,n},\hat{b}_{i,n}]} \int_{(\hat{a}_{j,n},\hat{b}_{j,n}]} \big( \mathds{F}_n\big( k_{i,j}(u,v)\big) - \mathds{F}_{i,n}(u)\mathds{F}_{j,n}(v) \big) \ \dif \mathds{G}_{j,n}(v)  \dif \mathds{G}_{i,n}(u) \\
& \hspace{5cm }- \frac{1}{\lambda} \iint_D \big(F(k_{i,j}(u,v)) - F_i(u)F_j(v)\big) P_i^{(1)} \otimes P_j^{(1)} (\dif u, \dif v) \Big| \\ 
&\leq \quad \Big| \int_{(a_i,b_i] \times (a_j,b_j]} \big( F(k_{i,j})(u,v) - F_{i}(u)F_{j}(v) \big) \ P^{(1)}_{j}\otimes P^{(1)}_{i}(\dif u, \dif v) \\
& \hspace{1cm} - \int_{(a_i,b_i] \times (a_j,b_j]} \big( \mathds{F}_n(k_{i,j})(u,v) - \mathds{F}_{i,n}(u)\mathds{F}_{j,n}(v) \big) \ \mathds{P}^{(1)}_{j,n}\otimes \mathds{P}^{(1)}_{i,n}(\dif u, \dif v)  \Big| \\
& \quad + \Big| \int_{(a_i,b_i] \times (a_j,b_j]} \big( \mathds{F}_n(k_{i,j})(u,v) - \mathds{F}_{i,n}(u)\mathds{F}_{j,n}(v) \big) \ \mathds{P}^{(1)}_{j,n}\otimes \mathds{P}^{(1)}_{i,n}(\dif u, \dif v) \\
& \hspace{1cm} - \int_{(\hat{a}_{i,n},\hat{b}_{i,n}] \times (\hat{a}_{i,n},\hat{b}_{i,n}]} \big( \mathds{F}_n(k_{i,j})(u,v) - \mathds{F}_{i,n}(u)\mathds{F}_{j,n}(v) \big) \ \mathds{P}^{(1)}_{j,n}\otimes \mathds{P}^{(1)}_{i,n}(\dif u, \dif v)  \Big| \\
& \quad + o(1) \\
&\leq \quad \Big(\big\|\mathds{F}_n - F\big\|_\infty + \big\|\mathds{F}_{i,n}\mathds{F}_{j,n}- F_{i}F_{j}\big\|_\infty \big) \\
& \quad + \Big| \int \ind_{(a_i,b_i]\times (a_j,b_j]}(u,v) \ \big(F(k_{i,j})(u,v)-F_i(u)F_j(v)\big) \ \big( P_i^{(1)} \otimes P_j^{(1)} - \mathds{P}_{i,n}^{(1)} \otimes \mathds{P}_{j,n}^{(1)} \big)(\dif u, \dif v)  \Big|  \\
& \quad + \big(\mathds{P}_{i,n}^{(1)} \otimes \mathds{P}_{j,n}^{(1)}\big)\big( \big\{ (a_i,b_i]\times (a_j,b_j] \big\} \triangle \big\{ (\hat{a}_{i,n},\hat{b}_{i,n}] \times (\hat{a}_{j,n},\hat{b}_{j,n}] \big\} \big) \\
& \quad + o(1)\,.
\end{align*}
The first two addends converge to zero almost surely as a consequence of the Glivenko--Cantelli theorem. The third term is $o_{\pr}(1)$, and we conclude that
\begin{align*}
    &\Big|\frac{\alpha_n + \beta_n}{\alpha_n}\int_{(\hat{a}_{i,n},\hat{b}_{i,n}]} \int_{(\hat{a}_{j,n},\hat{b}_{j,n}]} \big( \mathds{F}_n\big( k_{i,j}(u,v)\big) - \mathds{F}_{i,n}(u)\mathds{F}_{j,n}(v) \big) \ \dif \mathds{G}_{j,n}(v)  \dif \mathds{G}_{i,n}(u) \\
& \hspace{5cm }- \frac{1}{\lambda} \iint_D \big(F(k_{i,j}(u,v)) - F_i(u)F_j(v)\big) P_i^{(1)} \otimes P_j^{(1)} (\dif u, \dif v) \Big| \to 0
\end{align*}
as $n \to \infty$, in probability. The other addend in line \eqref{eq:eq} is treated analogously. \qed

\subsection*{Weak convergence of empirical processes}

\begin{definition}
For $n \in \Nbb$, let $X_n \colon \Omega \to (\ell_\infty^\kappa,d^\kappa_\infty)$ be $(\cf-\cp^\kappa)$-measurable, and let $X_0$ be $(\cf-\cb(\ell_\infty^\kappa))$-measurable. The sequence $\{X_n\}_{n \in \Nbb_+}$ is said to converge weakly to $X_0$, denoted by $X_n \rightsquigarrow  X_0$, if for every bounded continuous functional $f \colon (\ell_\infty^\kappa,d^\kappa_\infty) \to \Rbb$
\begin{align*}
\int^*f(X_n) \dif \pr \to \int f(X_0) \dif \pr \hspace{1cm } (n \to \infty)\,.
\end{align*}
Here, $\int^*f(X_n) \dif \pr = \inf\big\{ \int U \dif \pr \colon f(X_n) \leq U, \, U \colon (\Omega,\cf) \to (\overline{\Rbb},\cb(\overline{\Rbb})\big\}$. 
\end{definition}

As in the classical theory, this form of weak convergence is characterized as convergence of the finite dimensional distributions plus tightness.

\begin{definition}
The random element $X_0$ is tight if for every $\varepsilon > 0$ there exists a compact $K \subset (\ell_\infty^\kappa,d^\kappa_\infty)$ such that $\pr\{X_0 \in K\} > 1- \varepsilon$.

The sequence $\{X_n\}$ is asymptotically tight if for any $\varepsilon > 0$ there is a compact $K \subset (\ell_\infty^\kappa,d^\kappa_\infty)$ such that
$\inf_{\delta > 0} \ \liminf_{n \to \infty} \ \pr^*\big\{d^\kappa_\infty(X_n,K) < \delta\big\} > 1-\varepsilon$, where $\pr^*(B) := \inf \big\{\pr(A) \colon B \subset A, \ A \in \cf \big\}$ for $B \subset (\ell_\infty^\kappa,d_\infty)$. 

The sequence $\{X_n\}$ is asymptotically Borel measurable if
\begin{align*}
\int^* f(X_n) \dif \pr - \int_* f(X_n) \dif \pr \to 0 \hspace{1cm} (n \to \infty)\,,
\end{align*}
for every bounded continuous functional $f$ on $(\ell_\infty^\kappa,d_\infty^\kappa)$. Here, $\int_* f(X_n) \dif \pr := \sup\big\{ \int L \dif \pr \colon L \leq f, \ L \colon (\Omega,\cf) \to (\Rbb,\cb) \big\}$.
\end{definition}  

\begin{lemma}\label{lemma:1}
Let $\{P_n\}_{n \in \Nbb_+}$ be a sequence of probability measures on $(\ell_\infty^\kappa,\cp^\kappa)$ and $P_0$ a Borel probability measure on $(\ell_\infty^\kappa,\cb(\ell_\infty^\kappa))$, respectively. Furthermore, let $\{Y_n\}_{n \in \Nbb}$ be a sequence of random elements $Y_n \colon (\Omega,\cf) \to (\ell_\infty^\kappa,\cp)$ and $Y_0 \colon (\Omega,\cf) \to (\ell_\infty^\kappa,\cb(\ell_\infty^\kappa))$. Suppose that $P_n = \pr \circ Y_n^{-1}$ for all $n \in \Nbb$. If $\lim_{n \to \infty} \int^* f \dif P_n = \int f \dif P_0$ for every bounded continuous functional $f$ on $(\ell_\infty^\kappa,d_\infty^\kappa)$, then $Y_n \rightsquigarrow Y_0$.
\end{lemma}

\begin{proof}
First note that $\int^* f \dif P_n \to \int f \dif P_0$ implies
$\int^* f \dif P_n - \int_* f \dif P_n \to 0$ \citep[Lemma 1.3.8]{vandervaart1996}. The following chain of inequalities holds for any $n \in \Nbb_+$ and all bounded continuous functional $f$ on $\ell_\infty^k$:
\begin{align*}
\int^* f \dif P_n 
&= \inf \big\{ \int T \dif P_n \colon T \geq f\,, \ T \colon (\ell_\infty^\kappa,\cp^\kappa) \to (\Rbb,\cb)\big\}  \\
&\geq  \inf \big\{ \int T \circ Y_n \dif \pr \colon T\circ Y_n \geq f \circ Y_n\,, \ T \colon (\ell_\infty^\kappa,\cp^\kappa) \to (\Rbb,\cb)\big\}  \\
&\geq \inf \big\{ \int S \dif \pr \colon S \geq f \circ Y_n\,, \ S \colon (\Omega,\cf) \to (\Rbb,\cb)\big\}  \\
&\geq \sup \big\{ \int S \dif \pr \colon S \leq f \circ Y_n\,, \ S \colon (\Omega,\cf) \to (\Rbb,\cb)\big\}  \\
&\geq \sup \big\{ \int T \circ Y_n \dif \pr \colon T \circ Y_n \leq f \circ Y_n\,, \ T \colon (\ell_\infty^\kappa,\cp^\kappa) \to (\Rbb,\cb)\big\}  \\
&\geq \sup \big\{ \int T \dif P_n \colon T \circ Y_n \leq f \circ Y_n\,, \ T \colon (\ell_\infty^\kappa,\cp^\kappa) \to (\Rbb,\cb)\big\}  \\
&= \int_* f \dif P_n\,.
\end{align*} 
Therefore, $\int^* f \dif P_n \leq \int^* f(Y_n) \dif \pr \leq  \int_* f \dif P_n$, and $\lim_{n \to \infty} \int^* f(Y_n) \dif \pr$ exists and equals $\lim_{n \to \infty} \int^* f \dif P_n $. Hence, $Y_n \rightsquigarrow Y_0$.
\end{proof}

\begin{theorem}[\citet{vandervaart1996}, p. 35]\label{thm:k=1}
Let $\kappa=1$. If the sequence $\{X_n\}_{n \in \Nbb_+}$ is asymptotically tight and for any $t_1,\ldots,t_r \in \Rbb$ and $r \in \Nbb$
\begin{align*}
\big(\pi_{t_1}\circ X_n, \ldots, \pi_{t_r} \circ X_n \big) \rightsquigarrow \big(\pi_{t_1}\circ X_0, \ldots, \pi_{t_r} \circ X_0 \big)
\end{align*}
in $\Rbb^r$, then $X_n \rightsquigarrow X_0$ in $(\ell_\infty,d_\infty)$ and the limit $X_0$ is a Borel measurable tight random element. 
\end{theorem}

\subsection*{Empirical bootstrap fundamental theorem}

\begin{theorem}[\citet{vandervaart1996}, Thm. 23.9]\label{Thm:boostrap_fundamental}
Let $\mathds{D}$ be a normed space and denote by $\mathrm{BL}_1(\mathds{D})$ the set of bounded Lipschitz functions on $\mathds{D} \to [-1,1]$ with Lipschitz constant at most $1$. Suppose that the map $\phi \colon \mathds{D}_\phi \subset \mathds{D} \to \Rbb^\kappa$ is Hadamard differentiable at $\theta$ tangentially to a subspace $\mathds{D}_0$. Let $\hat{\theta}_n$ and $\hat{\theta}_n^*$ be maps with values in $\mathds{D}_\phi$ such that $n^{1/2}(\hat{\theta}_n - \theta) \rightsquigarrow T$, with a tight
$\mathds{D}_0$-valued random element $T$. Suppose that 
\begin{align*}
\sup_{h \in \mathrm{BL}_1(\mathds{D})} \Big| \ew \big[ h\big( n^{1/2}(\hat{\theta}_n^* - \hat{\theta}_n) \big) \bigm| \xi_1,\ldots,\xi_n \big] - \ew h(T) \Big| \to 0
\end{align*}
in outer probability and that $n^{1/2}(\hat{\theta}_n^* - \hat{\theta}_n)$ is asymptotically measurable. Then
\begin{align*}
\sup_{h \in \mathrm{BL}_1(\Rbb^\kappa)} \Big| \ew \big[ h\big( n^{1/2}\big(\phi(\hat{\theta}_n^*) - \phi(\hat{\theta}_n)\big) \big) \bigm| \xi_1,\ldots,\xi_n \big] - \ew h\big(D\phi(\theta)[T]\big) \Big| \to 0
\end{align*}
in outer probability.
\end{theorem}

\end{document}